%% file: main-arxiv.tex
\documentclass[a4paper,12pt]{article}

\usepackage[margin=1in]{geometry}
\usepackage[onehalfspacing]{setspace}

\usepackage[table]{xcolor}

\usepackage{amssymb}
\usepackage{amsthm}
\usepackage{algorithm}
\usepackage{algpseudocodex}
\usepackage{tikz}
\usepackage{pgfplots}
\pgfplotsset{compat=1.18}
\usetikzlibrary{math,calc,intersections,through,arrows.meta,external}
\usepackage{tkz-euclide}

\usepackage[labelfont=bf]{caption}
\usepackage[font=footnotesize]{subcaption}

\usepackage{mathtools}
\usepackage[hidelinks,
            colorlinks = false,
            allcolors=blue]{hyperref} 
\usepackage{nicefrac}
\usepackage{enumitem}
\usepackage{booktabs}
\usepackage{multirow}
\usepackage[misc]{ifsym}
\usepackage{soul}
\usepackage{rotating}

\usepackage[%
backend=biber, 
sorting=nyt, 
sortcites, 
safeinputenc, 
citestyle=authoryear, %
bibstyle=ext-authoryear, %
hyperref=true, 
maxbibnames=99, 
maxcitenames=2, 
uniquelist=false, 
url=true, 
isbn=false,
doi=true, 
giveninits=true,
date=year,
natbib
]%
{biblatex}

\DeclareNameAlias{sortname}{last-first} 

{\theoremstyle{definition} 
\newtheorem{definition}{Definition}}

\newtheorem{theorem}{Theorem}
\newtheorem{corollary}[theorem]{Corollary}
\newtheorem{lemma}[theorem]{Lemma}
\newtheorem{proposition}[theorem]{Proposition}

\newcommand{\T}{\mathsf{T}}

\renewcommand{\P}{\textsf{P}}
\newcommand{\NP}{\textsf{NP}}

\newcommand{\bigO}[1]{\mathcal{O}\left(#1\right)}
\newcommand{\oPoly}[1]{\ensuremath{\mathcal{O}(\mathsf{poly}( #1 ))}}
\newcommand{\codI}{\ensuremath{\langle I \rangle}}
\newcommand{\oStand}{\oPoly{\codI,\varepsilon^{-1}}}

\newcommand{\face}{\mathsf{face}}
\newcommand{\facets}{\mathsf{facet}}
\newcommand{\conv}{\mathsf{conv}}
\newcommand{\cone}{\mathsf{cone}}
\newcommand{\vertx}{\mathsf{vert}}
\newcommand{\rect}{\mathsf{rect}}
\newcommand{\halfsp}{\mathcal{H}}

\newcommand{\dobs}{d}
\newcommand{\solfeas}{X}
\newcommand{\solob}{Y}
\newcommand{\solnd}{Y_\mathrm{N}}

\newcommand{\minCone}{\mathbb{R}_{\geq0}^d}
\newcommand{\edgi}{Y^{+}}
\newcommand{\appri}{\mathcal{A}}
\newcommand{\apprii}[1]{\appri^{(#1)}}

\newcommand{\cptas}{cPTAS}
\newcommand{\cfptas}{cFPTAS}
\newcommand{\cbothptas}{c(F)PTAS}

\newcommand{\LB}{\mathrm{LB}}
\newcommand{\UB}{\mathrm{UB}}
\newcommand{\subdi}{\mathcal{S}}
\newcommand{\ideal}[1]{{#1}^*}

\newcommand{\WS}{\mathsf{WS}}
\newcommand{\hpO}{\mathsf{hpO}}
\newcommand{\hpOWS}{\hpO\text{-}\WS}

\newcommand{\oaa}{\texttt{OAA}}
\newcommand{\iaa}{\texttt{IAA}}

\newcommand{\opt}{\mathsf{opt}}
\newcommand{\Eapx}{E}
\newcommand{\genCone}{C}

\newcommand{\dist}{\mathrm{dist}}

\definecolor{cborange}{HTML}{E69F00}
\definecolor{cbskyblue}{HTML}{56B4E9}
\definecolor{cbgreen}{HTML}{009E73}
\definecolor{cbyellow}{HTML}{F0E442}
\definecolor{cbdarkblue}{HTML}{0072B2}
\definecolor{cbred}{HTML}{D55E00}
\definecolor{cbpurple}{HTML}{CC79A7}

\definecolor{lightgray}{gray}{0.925}

\newcommand{\yset}{
    \node(y1) at (3.0,3.2) {};
    \node(y2) at (3.1,4.8) {};
    \node(y3) at (3.8,3.5) {};
    \node(y4) at (3.4,4.3) {};
    \node(y5) at (5.0,5.1) {};
    \node(y6) at (2.1,3.8) {};
    \node(y7) at (2.9,1.6) {};
    \node(y8) at (5.5,1.4) {};
    \node(y9) at (1.9,2.3) {};
}

\begin{document}
\title{\vspace{-1cm}A Polynomial-Time Inner Approximation Algorithm for Multi-Objective and Parametric Optimization}

\author{Levin Nemesch\textsuperscript{\Letter}\footnote{RPTU Kaiserslautern-Landau, Department of Mathematics, Paul-Ehrlich-Str. 31, 67663 Kaiserslautern, Germany, email:\{l.nemesch,stefan.ruzika\}@math.rptu.de}, Stefan Ruzika\footnotemark[\value{footnote}]~,
     Clemens Thielen\footnote{Technical University of Munich, Campus Straubing for Biotechnology and Sustainability, Professorship of Optimization and Sustainable Decision Making, Am Essigberg~3, 94315~Straubing, Germany, email:\{clemens.thielen,alina.wittmann\}@tum.de}, Alina Wittmann\footnotemark[\value{footnote}]
}

\maketitle

\begin{abstract}

\singlespacing\small

In multi-objective optimization, computing the entire non-dominated set (also known as the Pareto front or the Pareto frontier) is often intractable.
However, for any multiplicative factor greater than one, an approximation set can be constructed in polynomial time for many problems.
In this paper, we use the concept of convex approximation sets:
Each point in the non-dominated set is approximated by a convex combination of images of solutions in such a set.
Convex approximation sets can be used to efficiently approximate multi-objective optimization problems as well as parametric optimization problems.
Recently, \citet{helfrich.ruzika.ea2024EfficientlyConstructing} presented a convex approximation algorithm that works in an adaptive fashion and runs faster than all previously existing algorithms.
We use a different approach for constructing an even more efficient adaptive algorithm for computing convex approximation sets of multi-objective mixed-integer linear programs.
Our algorithm is based on the skeleton algorithm for polyhedral inner approximation by \citet{csirmaz2021InnerApproximationAlgorithm}.
If the weighted sum scalarization can be solved exactly or approximately in polynomial time, our algorithm can find a convex approximation set for an approximation factor arbitrarily close to this solution quality.
We demonstrate that our new algorithm runs faster than the current state-of-the-art algorithm from \citet{helfrich.ruzika.ea2024EfficientlyConstructing} on instances of the multi-objective variants of the assignment problem, the knapsack problem, and the symmetric metric travelling salesman problem.

\medskip

\noindent
\textbf{Publishing:} This article is published with the same title and authors by the INFORMS Journal on Computing (IJOC), with doi:\;\href{https://doi.org/10.1287/ijoc.2025.1308}{\texttt{10.1287/ijoc.2025.1308}}. Cite the IJOC publication when using this article.

\smallskip


\smallskip

\noindent
\textbf{Funding:} This research was funded by the Deutsche Forschungsgemeinschaft (DFG, German Research Foundation) --- Project numbers 508981269 and GRK 2982, 516090167 ``Mathematics of Interdisciplinary Multiobjective Optimization''.

\smallskip

\noindent
\textbf{Repository:} Source code, instances and results for our computational study are provided in a git repository at: \href{https://github.com/INFORMSJoC/2025.1308}{\texttt{github.com/INFORMSJoC/2025.1308}}
\end{abstract}

\maketitle

\section{Introduction}\label{sec::intro}

In multi-objective optimization, the goal is to optimize several conflicting objectives simultaneously. As a result, a solution that performs well in one objective may perform poorly in another, and no single solution is best in all objectives. Instead, the aim is to identify the \emph{non-dominated set} (also called the \emph{Pareto front} or the \emph{Pareto frontier}), which consists of all \emph{non-dominated images} --- objective vectors of feasible solutions for which no other objective vector of a feasible solution is at least as good in every objective and strictly better in at least one. A solution that corresponds to a non-dominated image is called \emph{efficient}. In many cases, one is not only interested in describing the non-dominated set itself but also in finding a set of efficient solutions that map to it.

In this paper, we focus on \emph{multi-objective mixed-integer linear programs} (MOMILPs) with a bounded, non-negative image set.
It is well known that the number of non-dominated images in such problems can be exponential in the problem size \citep{ehrgott2005MulticriteriaOptimization}, or even infinite.
Furthermore, even deciding whether a given image belongs to the non-dominated set can already be \NP-hard \citep{serafini1987ConsiderationsComputationalComplexity}.
Therefore, for many problems, finding polynomial-time algorithms to compute the entire non-dominated set is a futile effort.

\smallskip

An alternative to computing the entire non-dominated set can be found in the concept of \emph{multi-objective approximation}:
Here, a multiplicative approximation factor~$1+\varepsilon$ with $\varepsilon>0$ is given, and the goal is to find a $(1+\varepsilon)$-\emph{approximation set}.
In a $(1+\varepsilon)$-approximation set, every feasible solution is $(1+\varepsilon)$-approximated, i.e., there exists a solution in the approximation set that is worse by at most a factor $(1+\varepsilon)$ in every objective.
A \emph{$(1+\varepsilon)$-approximation algorithm} is an algorithm that computes a $(1+\varepsilon)$-approximation set in polynomial time.
Unlike heuristic methods, approximation algorithms aim not only for computational efficiency but also for a provable guarantee on the quality of the output.

Under rather mild assumptions, a key result was shown in \citet{papadimitriou.yannakakis2000ApproximabilityTradeoffsOptimal}.
For every~$\varepsilon>0$, a multi-objective optimization problem admits a $(1+\varepsilon)$-approximation set of cardinality polynomial in the input and $\nicefrac{1}{\varepsilon}$.
The proof is constructive:
A polynomial-cardinality grid is constructed in~$\mathbb{R}^d$  (where~$d$ denotes the number of objectives) that contains all  images of feasible solutions.
For each cell in this grid, at most one solution needs to be in an approximation set.
Most of the literature on multi-objective approximation uses this grid in some way --- a comprehensive survey is given in \citet{herzel.ruzika.ea2021ApproximationMethodsMultiobjective}.
However, while being a relevant tool for theoretical purposes, this grid is rather inconvenient in practice.
There are two reasons for this.
First, while the cardinality of the grid is polynomial in the problem size, it often vastly overestimates the number of solutions needed in an approximation set \citep{helfrich.ruzika.ea2024EfficientlyConstructing}.
Second, the numerical values that appear in the construction of the grid have rather high encoding lengths, which easily leads to numerical issues in fixed precision arithmetic.
Therefore, finding grid-agnostic approaches is an important step in making multi-objective approximation algorithms viable in practice.

\smallskip

This motivates to look at the relaxed notion of \emph{multi-objective convex approximation}.
First investigated in \citet{diakonikolas.yannakakis2008SuccinctApproximateConvex}, a \emph{(multi-objective) convex approximation set} is a set of solutions such that each feasible image is approximated by a convex combination of images of solutions in the set.
As demonstrated in \citet{helfrich.ruzika.ea2024EfficientlyConstructing}, it is possible to develop convex approximation algorithms that work in an adaptive fashion and avoid many of the aforementioned issues that appear when directly using the grid from \citet{papadimitriou.yannakakis2000ApproximabilityTradeoffsOptimal}.

Convex approximation relaxes the notion of approximation and offers valuable applications.
The two most prominent use cases are the approximation of \emph{parametric optimization problems} (where linear combinations of objectives are considered) and of convex multi-objective optimization problems:
For a parametric optimization problem, every multi-objective convex approximation algorithm is a (parametric) approximation algorithm, even if the underlying problem structure is non-convex or combinatorial~\citep{helfrich.herzel.ea2022ApproximationAlgorithmGeneral}.
For a convex multi-objective optimization problem, a convex approximation set needs far fewer solutions than a conventional approximation set since it can take advantage of the convex problem structure.%

\smallskip

An early convex approximation algorithm that uses so-called \emph{weighted sum scalarizations} is outlined in \citet{papadimitriou.yannakakis2000ApproximabilityTradeoffsOptimal}.
In \citet{diakonikolas2011ApproximationMultiobjectiveOptimization} several advanced convex approximation algorithms are developed.
One of them, the so-called \emph{Chord algorithm} (also subject of \citet{daskalakis.diakonikolas.ea2016HowGoodChord}), can be seen as an approximate variant of the well-known dichotomic approach from bi-objective optimization \citep{aneja.nair1979BicriteriaTransportationProblem}.
Similarly, another adaptation of the dichotomic approach is developed in \citet{bazgan.herzel.ea2022ApproximationAlgorithmGeneral} for the approximation of parametric optimization problems, but it can also be used to compute convex approximation sets.

Based on a parametric approximation algorithm from \citet{helfrich.herzel.ea2022ApproximationAlgorithmGeneral},  a generic convex approximation algorithm that generalizes the so-called \emph{dual Benson algorithm} from multi-objective optimization is developed in \citet{helfrich.ruzika.ea2024EfficientlyConstructing}.
This algorithm works in an adaptive fashion and interacts only with selected points from a special grid.
While the construction of this grid is somewhat similar to that of the grid from \citet{papadimitriou.yannakakis2000ApproximabilityTradeoffsOptimal}, it partitions the so-called \emph{weight-space} instead.
Although the worst-case running time of the new algorithm is worse than that of the purely grid-based algorithm from \citet{helfrich.herzel.ea2022ApproximationAlgorithmGeneral}, it runs much faster in a computational study.
This highlights the practical viability of approaches that minimize interactions with the grid from \citet{papadimitriou.yannakakis2000ApproximabilityTradeoffsOptimal} or similar grid constructions, even if the worst-case running time worsens because of this.

\smallskip

We build upon this result and provide a new algorithm for convex approximation sets that works completely grid-agnostic.
We do this by adapting a generic \emph{inner approximation} method from exact multi-objective optimization.
Note that the term ``approximation'' in the context of inner approximation methods differs from how ``approximation'' is used above.
An \emph{inner approximation polyhedron} denotes a polyhedron that is contained in the so-called \emph{Edgeworth-Pareto hull} of the non-dominated images \citep{csirmaz2021InnerApproximationAlgorithm}.
However, no guarantee concerning an approximation factor such as $1+\varepsilon$ is ensured.

The most well-known algorithms for computing inner (and, also, outer) approximation polyhedra are those from the ``Benson family''.
The first outer approximation algorithm for \emph{multi-objective linear programs} (MOLPs) is introduced in \citet{benson1998OuterApproximationAlgorithm}.
On the basis of geometric duality for MOLPs \citep{heyde.lohne2008GeometricDualityMultiple}, a dual variant of this algorithm is developed in \citet{ehrgott.lohne.ea2012DualVariantBenson}.
Due to geometric duality, this algorithm can be seen as an inner approximation method \citep{csirmaz2021InnerApproximationAlgorithm}.
Improved variants for Benson's algorithm and the dual Benson algorithm are presented in \citet{hamel.lohne.ea2014BensonTypeAlgorithms} and in \citet{bokler2018OutputsensitiveComplexityMultiobjective}, and some are implemented and compared in a computational study in \citet{lohne.weissing2017VectorLinearProgram}.
Variants that use an absolute notion of approximation to speed up practical running times are described in literature (for an overview, see \citet{lohne.rudloff.ea2014PrimalDualApproximation}), but do not ensure a polynomial worst-case running time.
Both \citet{bokler.mutzel2015OutputSensitiveAlgorithmsEnumerating} and \citet{borndorfer.schenker.ea2016PolySCIP} independently show how to apply the dual Benson algorithm to \emph{multi-objective integer linear programs} (MOILPs), and \citet{bokler.nemesch.ea2023PaMILOSolverMultiobjective} extend this to also include MOMILPs.

The relation between inner and outer approximation algorithms is studied in \citet{csirmaz2021InnerApproximationAlgorithm},
and skeleton algorithms for both are presented.
These skeleton algorithms are parametrized by the choice of so-called \emph{separating oracles}.
The outer approximation skeleton algorithm is used in \citet{bokler.parragh.ea2024OuterApproximationAlgorithm} to develop the first outer approximation algorithm for MOMILPs.
Independently, an algorithm for MOMILPs that fits into the inner approximation skeleton algorithm is presented in \citet{przybylski.klamroth.ea2019SimpleEfficientDichotomic}.

\subsection{Our Contribution}\label{sec::intro::contribution}

Our central contribution is to combine the generic inner approximation algorithm from \citet{csirmaz2021InnerApproximationAlgorithm} with the concept of multi-objective convex approximation.
Our algorithm can also be seen as a descendant of the well-known dual Benson algorithm \citep{ehrgott.lohne.ea2012DualVariantBenson}.
It computes convex approximation sets adaptively without accessing the grid from \citet{papadimitriou.yannakakis2000ApproximabilityTradeoffsOptimal}.
This sets it apart from most, if not all, previous convex approximation algorithms.
These algorithms are either limited in the number of objectives \citep{diakonikolas2011ApproximationMultiobjectiveOptimization,daskalakis.diakonikolas.ea2016HowGoodChord,bazgan.herzel.ea2022ApproximationAlgorithmGeneral}, use a grid in the so-called \emph{weight space} \citep{papadimitriou.yannakakis2000ApproximabilityTradeoffsOptimal,helfrich.herzel.ea2022ApproximationAlgorithmGeneral}, or need to compute at least some vertices of such a grid \citep{helfrich.ruzika.ea2024EfficientlyConstructing}.
The assumptions we impose on the considered problems are rather weak and match those used in most other literature on multi-objective approximation \citep{herzel.ruzika.ea2021ApproximationMethodsMultiobjective}.
We only describe our algorithm for the minimization case, but it can also be used for the maximization case.
In Appendix~\ref{sec::innerapx::generalization}, we briefly outline an even more general setting where the algorithm is applicable.

We conduct a computational study to evaluate the performance of our algorithm and compare it to the state-of-the-art algorithm from \citet{helfrich.ruzika.ea2024EfficientlyConstructing}.
The results show that, although our algorithm has a higher worst-case running time in theory, it runs much faster than the algorithm from~\citet{helfrich.ruzika.ea2024EfficientlyConstructing}.

\section{Preliminaries}\label{sec::prelims}

For $n\in\mathbb{N}_{>0}$, we write $[n]\coloneqq \{1,\ldots,n\}$.
As a shorthand, we write \oPoly{x} to indicate that a polynomial bound in~$x$ exists.
We use the notation $\langle x\rangle$ to denote the binary encoding length of an object~$x$, where~$x$ can be a problem instance or a mathematical object (e.g.\ a number, a vector, or a matrix).
Throughout this paper, addition of sets is always done via the Minkowski sum:
For $S_1,S_2\subseteq\mathbb{R}^n$, we write $S_1+S_2\coloneqq \{s_1+s_2: s_1\in S_1,s_2\in S_2\}$.
When comparing two vectors $u,v\in\mathbb{R}^n$, $u\leq v $ ($u\geq v$) denotes that $u_i\leq v_i$ ($u_i\geq v_i$) for all $i\in[i]$.
For a set $S\subseteq\mathbb{R}^n$, $\conv(S)$ denotes its \emph{convex hull} and $\cone(S)$ denotes its \emph{conic hull}.

\smallskip

We now recall some basic concepts for polyhedra that are used throughout this paper.
For an in-depth introduction, we refer to \citet{ziegler1995LecturesPolytopes} and \citet{grotschel.lovasz.ea1988RationalPolyhedra}.

For $w\in\mathbb{R}^{n}\setminus \{{0}\}$ and $c\in\mathbb{R}$, the set $\left\{z \in\mathbb{R}^{n} : w^\T z\geq c\right\}$ is called a \emph{halfspace}, and its boundary $\left\{z \in\mathbb{R}^{n} : w^\T z= c\right\}$ is called a \emph{hyperplane}.
A \emph{polyhedron} is the intersection of a finite number of halfspaces.
A \emph{face}~$F$ of a polyhedron~$\mathcal{P}$ is either~$\mathcal{P}$ itself, or a set of the form $F=\mathcal{P}\cap h$, where~$h$ is a hyperplane that intersects only the boundary of $\mathcal{P}$.
If $F$ has dimension~$0$ it is called a \emph{vertex} of~$\mathcal{P}$, and if it has one dimension less than~$\mathcal{P}$ it is called a \emph{facet} of~$\mathcal{P}$.
For~$\mathcal{P}$, a halfspace~$H$ is \emph{valid} if $\mathcal{P}\subseteq H$, and \emph{supporting} if it is valid and its boundary intersects~$\mathcal{P}$.
A halfspace is \emph{facet-supporting} if its boundary contains a facet.
The set $\vertx(\mathcal{P})$ denotes the  set of vertices of the polyhedron~$\mathcal{P}$, the set $\face(\mathcal{P})$ denotes its faces, and the set $\facets(\mathcal{P})$ its facets.
With $\halfsp(\mathcal{P})$ we denote the set of facet-supporting halfspaces.
The \emph{recession cone} of~$\mathcal{P}$ is the set $\left\{r\in\mathbb{R}^n: \forall p\in\mathcal{P},a\geq0: p+ar\in\mathcal{P}\right\}$.

A polyhedron that is represented as an intersection of halfspaces is called a polyhedron in \emph{$\mathcal{H}$-representation}.
Alternatively, a polyhedron can be represented in a \emph{$\mathcal{V}$-representation}:
For finite sets $S_1,S_2\subseteq\mathbb R^{n}$, $\mathcal{P}=\conv{(S_1)}+\cone{(S_2)}$ defines a polyhedron.
Every convex polyhedron can be represented in an $\mathcal{H}$-representation and in a $\mathcal{V}$-representation.
A polyhedron is called \emph{rational} if there exists a representation in which all appearing numbers are in~$\mathbb{Q}$.
Generating one type of representation from the other is the goal of the well-known \emph{vertex} and \emph{facet enumeration} problems.
A vertex enumeration generates a $\mathcal{V}$-representation from an $\mathcal{H}$-representation, a facet enumeration an $\mathcal{H}$-representation from a $\mathcal{V}$-representation.

\smallskip

Next, we introduce the relevant concepts from multi-objective optimization.
For $\dobs,m,n\in\mathbb{N}_{>0}$, let $Q\in\mathbb{Q}^{d\times n}$, $A\in\mathbb{Q}^{m\times n}$, $b\in\mathbb{Q}^m$, and $n_\mathbb{R},n_\mathbb{Z}\in\mathbb{N}_{\geq0}$ with $n_\mathbb{R}+n_\mathbb{Z}=n$.
Then, a \emph{multi-objective mixed-integer linear program} (MOMILP) is a problem of the form
\[
\setlength\arraycolsep{1pt}
\begin{array}{lrl}
    \min\hspace*{10pt}  & Qx \\
    \mathrm{s.t.} & Ax & \geq b ,\\
    & \multicolumn{2}{r}{x \in\mathbb{R}^{n_\mathbb{R}}\times\mathbb{Z}^{n_\mathbb{Z}}.}
\end{array}
\]
As shorthands, $f(x)\coloneqq Qx$ denotes the \emph{objective function}, and $\solfeas\subseteq\mathbb{R}^n$ denotes the set of \emph{feasible solutions}.
For a feasible solution~$x$, we call the point~$f(x)\in\mathbb{R}^d$ its \emph{image}, and denote the set of all images by $\solob \coloneqq \left\{f(x) : x\in\solfeas\right\}\subseteq\mathbb{R}^d$.
Each component~$f_i$ of~$f$ (with $i\in[\dobs]$) is called an \emph{objective}.
The $\dobs$-dimensional space containing the images is called the \emph{objective space}.
A MOMILP is called a \emph{multi-objective linear program (MOLP)} if $n_\mathbb{Z}=0$, and a \emph{multi-objective integer linear program (MOILP)} if $n_\mathbb{R}=0$.
Note that our algorithm does not require a problem to be explicitly encoded as above.
For example, a shortest-path problem is a MOMILP as it can be written as such, but it might be encoded as a graph with edge costs instead.

To allow the existence of a convex approximation set for a MOMILP, it is assumed that $\solob\subseteq\mathbb{R}_{\geq0}^d$ \citep{papadimitriou.yannakakis2000ApproximabilityTradeoffsOptimal}.
Furthermore, the number~$d$ of objectives is assumed to be a constant and not part of the input.
These assumptions are regularly used in the literature on multi-objective (convex) approximation \citep{herzel.ruzika.ea2021ApproximationMethodsMultiobjective}.
We also assume that $\solfeas\neq\emptyset$.

For two points $u,v\in \mathbb{R}^{\dobs}$, $u$ \emph{dominates}~$v$ if $v\in \{u\}+\minCone$ and $u\neq v$.
In this setting, $\minCone$ is also known as the \emph{domination cone}.
An image~$y\in\solob$ is called \emph{non-dominated} if there is no $y'\in\solob$ that dominates~$y$.
The \emph{non-dominated set}~$\solnd$ is the set of all non-dominated images.
A solution~$x\in\solfeas$ is called \emph{efficient} (or \emph{Pareto-optimal}) if $f(x)\in\solnd$.
Commonly, the goal of a multi-objective optimization problem is to find a set of efficient solutions so that the corresponding images constitute~$\solnd$.

The \emph{Edgeworth-Pareto hull} of a MOMILP is the set $\edgi\coloneqq\conv\left(\solob\right)+\minCone$.
It is a rational polyhedron, and the encoding length of each vertex of~$\edgi$ is in $\oPoly{\codI}$ \citep{bokler.parragh.ea2024OuterApproximationAlgorithm}.
Hence, all facet-supporting halfspaces can also be encoded in polynomial length \citep{grotschel.lovasz.ea1988RationalPolyhedra}.
Furthermore, for each vertex~$v$ of~$\edgi$, there is a solution~$x\in\solfeas$ with polynomial encoding length such that $f(x)=v$ \citep{bokler.parragh.ea2024OuterApproximationAlgorithm}.
A solution whose image is on the boundary of~$\edgi$ is called a \emph{(weakly) supported solution}, and a solution whose image is a vertex of~$\edgi$ is called an \emph{extreme supported solution}.
For ease of notation, an image~$f(x)$ is also called (extreme) supported if~$x$ is (extreme) supported.

\subsection{Multi-Objective Approximation}\label{sec::prelims::apx}

We briefly introduce the concept of multi-objective approximation, a comprehensive survey is given in \citet{herzel.ruzika.ea2021ApproximationMethodsMultiobjective}.
From now on, we assume that a MOMILP instance~$I$ and an $\varepsilon\in\mathbb{R}_{>0}$ are given if not stated otherwise.

\begin{definition}
    For two points $u,v\in\mathbb{R}_{\geq0}^d$, $u$ \emph{$(1+\varepsilon)$-approximates}~$v$ if $u \leq (1+\varepsilon)\cdot v$.
    Similar, for two solutions $x,x'\in\solfeas$, we say that~$x$ \emph{$(1+\varepsilon)$-approximates}~$x'$ if $f(x)\leq (1+\varepsilon) \cdot f(x')$.
\end{definition}

\begin{definition}
    A set $R\subseteq\solfeas$ is called a \emph{$(1+\varepsilon)$-approximation set} if, for each~$x\in\solfeas$, there exists a solution~$x^*\in R$ with $f(x^*)\leq(1+\varepsilon)\cdot f(x)$.
\end{definition}

In \citet{papadimitriou.yannakakis2000ApproximabilityTradeoffsOptimal} it is shown that, under rather mild assumptions, there always exists a $(1+\varepsilon)$-approximation set with cardinality in \oStand.
A relaxed notion of multi-objective approximation is introduced in \Citet{diakonikolas.yannakakis2008SuccinctApproximateConvex}:
\begin{definition}\label{def::prelims::convexApx}
    A set $R\subseteq\solfeas$ is called a \emph{$(1+\varepsilon)$-convex-approximation set} if, for each~$x\in\solfeas$, there is a $v\in\conv\left\{f(x^*) : x^*\in R\right\}$ such that $v\leq(1+\varepsilon)\cdot f(x)$.
\end{definition}

Note that the point~$v$ in Definition~\ref{def::prelims::convexApx} is not necessarily an image in~$Y$.

\begin{definition}
    A \emph{$(1+\varepsilon)$-convex approximation algorithm} is an algorithm that takes as input a MOMILP instance~$I$ and outputs a $(1+\varepsilon)$-convex approximation set.
    Its running time must be bounded in \oPoly{\codI}.
\end{definition}

The concept of \emph{approximation schemes} is defined analogously to the case of single-objective optimization (see \citet{vazirani2001ApproximationAlgorithms}).

\begin{definition}
    A \emph{multi-objective convex approximation scheme} is an algorithm that takes as input a MOMILP instance~$I$ and an $\varepsilon\in\mathbb{Q}_{>0}$, and outputs a $(1+\varepsilon)$-convex approximation set.
    It is called a \emph{multi-objective convex polynomial-time approximation scheme (\cptas)} if the running time for every fixed $\varepsilon$ is in \oPoly{\codI}.
    Furthermore, it is called a  \emph{multi-objective convex fully polynomial-time approximation scheme (\cfptas)} if the running time for every~$\varepsilon$ is in $\oPoly{\langle I\rangle,\varepsilon^{-1}}$.
\end{definition}

In the case of a single-objective optimization problem, the multi-objective notion of (convex) approximation is equivalent to the conventional notion of approximation.
PTAS and FPTAS then denote traditional single-objective approximation schemes.

\smallskip

The problem of constructing a $(1+\varepsilon)$-convex approximation set can also be interpreted as a polyhedral approximation problem.
For the Edgeworth-Pareto hull~$\edgi$, we define the polyhedron $\edgi_{1+\varepsilon}\coloneq\left\{(1+\varepsilon) v : v\in\edgi\right\}$.
It can be easily seen that $\minCone$ is the recession cone of $\edgi_{1+\varepsilon}$.
The polyhedron $\edgi_{1+\varepsilon}$ motivates a different characterization of a $(1+\varepsilon)$-convex approximation sets.
Similar to the construction of the Edgeworth-Pareto hull, a set of solutions $R\subseteq\solfeas$ can be used to determine a polyhedron
$\appri\coloneqq\conv\left\{f(x) : x\in R\right\}+\minCone$
in objective space.
If $\appri\subseteq\edgi_{1+\varepsilon}$ holds, then~$R$ must be a $(1+\varepsilon)$-convex approximation set.
This is a key insight that motivates our algorithm in Section~\ref{sec::innerapx}.
Proposition~\ref{proposition::prelims::edgiInAppri} formalizes the concept:

\begin{proposition}\label{proposition::prelims::edgiInAppri}
    Let $R\subseteq\solfeas$ be finite and let~$\appri$ be the resulting polyhedron in objective space.
    $R$ is a $(1+\varepsilon)$-convex approximation set for~$I$ if and only if $\edgi_{1+\varepsilon}\subseteq\appri$.
\end{proposition}
\begin{proof}{Proof}
    First, assume that~$R$ is a $(1+\varepsilon)$-convex approximation set.
    For every point $v_{1+\varepsilon}\in\edgi_{1+\varepsilon}$, there must be a point~$v$ in $\edgi$ such that $v_{1+\varepsilon}=(1+\varepsilon) v$.
    As~$\edgi$ is a polyhedron, there are images $y_{1},\ldots,y_{k}\in Y$, scalar values $\ell_1,\ldots,\ell_{k}\geq0$ with $\sum_{i=1}^{k}\ell_i=1$, and a $w\in \minCone$ such that $v=\sum_{i=1}^{k}\ell_i y_i+w$.
    Hence,~$v_{1+\varepsilon}$ can be written as $(1+\varepsilon) v=\sum_{i=1}^m\ell_i (1+\varepsilon) y_i+(1+\varepsilon) w$.

    Since~$R$ is a convex approximation set, $(1+\varepsilon) y_i$ must be in $\conv\left\{f(x) : x\in R\right\}+\minCone=\appri$ for every $i\in[k]$.
    Furthermore, as~$\appri$ is convex, every convex combination of $y_{1},\ldots,y_{k}$ must be in~$\appri$, too.
    The vector~$(1+\varepsilon) w$ is in~$\minCone$.
    As a consequence, the point~$v_{1+\varepsilon}$ must be in $\appri+\minCone=\appri$.

    Conversely, assume that $\edgi_{1+\varepsilon}\subseteq\appri$.
    Then, for each point~$v\in\edgi$, the point~$(1+\varepsilon) v$ is in~$\edgi_{1+\varepsilon}$ and, thus, contained in~$\appri$.
    This includes all images, as $Y\subseteq\edgi$.
    Hence, for each~$y\in \solob$, we have $(1+\varepsilon) y\in\appri=\conv\left\{f(x) : x\in R\right\}+\minCone$.
    By Definition~\ref{def::prelims::convexApx}, $R$ is a convex approximation set.
\end{proof}

In \citet{papadimitriou.yannakakis2000ApproximabilityTradeoffsOptimal}, a special grid is constructed to show the existence of $(1+\varepsilon)$-approximation sets of polynomial cardinality for each~$\varepsilon\in\mathbb{Q}_{>0}$.
We make use of this grid in Lemma~\ref{lemma::innerapx::onceperS} in our running time analysis in Section~\ref{sec::innerapx}.
Note that we tailor its construction to our use case and slightly deviate from the original construction in~\cite{papadimitriou.yannakakis2000ApproximabilityTradeoffsOptimal}.

The grid consists of a set of \emph{hyperrectangles}.
A \emph{hyperrectangle} is a set of the form $\rect(u,v)\coloneqq\left\{x\in\mathbb{R}^n : u\leq x \leq v\right\}$ for two points $u,v\in\mathbb{R}^n$ with $u\leq v$.
We call the point $u$ the \emph{minimal vertex} of $\rect(u,v)$.

Let $V\subseteq\mathbb{Q}_{\geq0}^\dobs$ be a finite subset of images from $\solob$, and let~$p$ be a bound on the encoding length of each image~$y\in V$.
The idea of the grid from \citet{papadimitriou.yannakakis2000ApproximabilityTradeoffsOptimal} is to construct a set of hyperrectangles so that each image from~$V$ is contained in at least one hyperrectangle, and so that one image per hyperrectangle is sufficient to $(1+\varepsilon)$-approximate all other images contained in the same hyperrectangle.

The first step in constructing this set is to find some bounds on the objective values that can appear in each image~$y\in V$.
Because of the encoding length bound~$p$,
each value~$v_{i\in[\dobs]}$ must either be exactly~$0$ or be between the lower bound $\LB=2^{-p}$ and the upper bound $\UB=2^p$.
Let $\kappa\coloneq\max\left\{k\in\mathbb{N}: \LB\cdot(1+\varepsilon)^k<\UB\right\}$.
Then, the following set of hyperrectangles spans the grid:

\begin{definition}\label{def::prelims::grid}
We define the set $\subdi$ of hyperrectangles as
    \[
        \subdi\coloneq\left\{
        \rect\left(s,(1+\varepsilon)s\right)\subset\mathbb{R}^d: \forall i\in[\dobs]\enspace s_i = 0 \vee s_i=\LB(1+\varepsilon)^k \text{ for } k=0,\ldots,\kappa
        \right\}.
    \]
\end{definition}

As shown in \citet{papadimitriou.yannakakis2000ApproximabilityTradeoffsOptimal}, each image from~$V$ is contained in at least one hyperrectangle of~$\subdi$ and the cardinality of~$\subdi$ is in $\bigO{{\left(\varepsilon^{-1}\log{\frac{UB}{LB}}\right)}^{\dobs}}\subseteq\bigO{{(\varepsilon^{-1} p)}^{\dobs}}$.
The number of hyperrectangles that contain a non-dominated image is further bounded by $\bigO{{(\varepsilon^{-1} p)}^{\dobs-1}}$ \citep{vassilvitskii.yannakakis2005EfficientlyComputingSuccinct}.
We note that the same reasoning can be applied to bound the number of supported images, but omit a formal proof here.
An additional property of $\subdi$ is needed in this paper:

\begin{lemma}\label{lemma::grid::alphaprops}
    For each $S\in\subdi$, let~$\ideal{s}$ be its minimial vertex. Then, for all $s\in S$ and $w\in \minCone$, it holds that $\ideal{s}\leq s \leq (1+\varepsilon) \ideal{s}$ and $w^\T \ideal{s}\leq w^\T s \leq w^\T \left((1+\varepsilon)\ideal{s}\right)$.
\end{lemma}
\begin{proof}{Proof}
    The first property follows directly from the definition of a hyperrectangle and the construction of~$\subdi$.
    Then, for each~$i\in[\dobs]$, it holds that $w_i\geq 0$ and, thus, $w_i\ideal{s}_i\leq w_i s_i \leq w_i\left((1+\varepsilon)\ideal{s}_i\right)$.
\end{proof}

\section{The Inner Approximation Algorithm}\label{sec::innerapx}

In this section, assume that a MOMILP instance $I$ and some $\varepsilon\in\mathbb{R}_{>0}$ are given (if not stated otherwise).
Our algorithm is based on the generic inner approximation algorithm described in \citet{csirmaz2021InnerApproximationAlgorithm}.
We first outline the original exact algorithm and then present how we adapt it for convex approximation.
Note that we slightly change the description of the original algorithm from the one provided in \citet{csirmaz2021InnerApproximationAlgorithm} to tailor it to our use case.
A central component of the algorithm is the following oracle:

\begin{definition}\label{def::planeseporacle}
    A \emph{(hyper-)plane separating oracle} $\hpO$ for a MOMILP instance~$I$ is a black box algorithm that takes as input a halfspace $H=\left\{z\in\mathbb{R}^{\dobs} : w^\T z\geq c\right\}$ with $w\in\minCone$ and returns a tuple that is either:
        \begin{itemize}[noitemsep]
            \item $\left(\textit{inside}, \emptyset\right)$ if it holds that $f(x)\in H$ for all $x\in\solfeas$, or
            \item $\left(\textit{not inside}, x\right)$, where $x\in\solfeas$ with $\langle x\rangle\in\oPoly{\langle I \rangle}$ and $f(x)\notin H$.
        \end{itemize}
\end{definition}

For ease of use, we say in the first case that a plane separating oracle returns no solution and, in the second case, that it returns the solution~$x$.
Compared to \citet{csirmaz2021InnerApproximationAlgorithm}, our plane separating oracles are defined slightly different.
First, we require bounds on the encoding length of the returned solutions.
Second, in \citet{csirmaz2021InnerApproximationAlgorithm}, every returned solution must be supported or extreme supported.
For our algorithm, the weaker condition from Definition~\ref{def::planeseporacle} is already sufficient.
We discuss the implications of using stronger oracles at the end of this section.

\smallskip

The goal of the inner approximation algorithm from \citet{csirmaz2021InnerApproximationAlgorithm} is to find $\edgi$.
This is done by starting with an initial inner approximation polyhedron that is contained in $\edgi$ and then iteratively refining this polyhedron until it coincides with $\edgi$.
The initial polyhedron~$\apprii{1}$ is constructed by picking a solution $x^\mathrm{init}\in\solfeas$ and setting $\apprii{1}\coloneq f(x^\mathrm{init})+\mathbb{R}_{\geq0}^d$.
For ease of use, we assume that this initial solution is given as part of the input.
After $\apprii{1}$ is constructed, the algorithm works iteratively, beginning with $k=1$.
In iteration $k$, all facet-supporting halfspaces of $\apprii{k}$ are computed via facet enumeration.
For each halfspace $H\in\mathcal{H}(\apprii{k})$, the plane separating oracle is called.
Halfspaces for which the oracle has already been called in a previous iteration can be skipped.
If, for any halfspace, the oracle returns a solution~$x\in\solfeas$, the inner approximation polyhedron is updated to $\apprii{k+1}\coloneq \conv\left(\{f(x)\} \cup \apprii{k}\right)+\minCone$, $k$ is incremented, and the next iteration begins.
Due to its construction, $\apprii{k}\subseteq\edgi$ holds in each iteration.
If the oracle returns no solution for all halfspaces of $\apprii{k}$, then it also holds that $\edgi\subseteq\apprii{k}$ and the algorithm terminates.
The final approximation polyhedron is denoted by~$\appri$.

A detailed analysis of the correctness and running time of the exact inner approximation algorithm with stronger plane separating oracle is given in \citet{csirmaz2021InnerApproximationAlgorithm}. We refrain from restating it here.

\smallskip

Adapting the inner approximation algorithm for convex approximation is accomplished through a straightforward modification.
In each iteration, instead of calling the plane separating oracle with the facet-supporting halfspaces directly, each halfspace is modified first.
Let $H=\left\{z\in\mathbb{R}^{\dobs} : w^\T z\geq c\right\}$ be the halfspace that is currently investigated by the algorithm.
Instead of calling the plane separating oracle for~$H$, it is called for the halfspace $H_{1+\varepsilon}\coloneqq\left\{z\in\mathbb{R}^{\dobs} : (1+\varepsilon)w^\T z\geq c\right\}$.
The idea is to simulate a call for~$H$ to a plane separating oracle of the polyhedron~$\edgi_{1+\varepsilon}$ instead of~$\edgi$.
To demonstrate why this modification can be seen as such an oracle, the inequality defining $H_{1+\varepsilon}$ can be rewritten into $w^\T((1+\varepsilon) z)\geq c$.
A solution can now only be returned if, for its image $y$, it holds that $(1+\varepsilon)y\notin H$.
As a consequence, the algorithm terminates as soon as $\edgi_{1+\varepsilon}\subseteq\apprii{k}$ instead of  $\edgi\subseteq\apprii{k}$.
The new algorithm is formally stated in Algorithm~\ref{algo::innerapx}.
Figure~\ref{fig::innerapx::run} illustrates Algorithm~\ref{algo::innerapx} on a small example.

\input{figures/innerapx.tex}

\input{figures/innerapx-example.tex}

We now show that Algorithm~\ref{algo::innerapx} is a $(1+\varepsilon)$-convex approximation algorithm if an initial solution can be provided and if each call to the plane separating oracle has running time in $\oPoly{\langle I, H \rangle}$ for every halfspace $H=\{z\in\mathbb{R}^d : w^\T z\geq c\}$ with $w\in\mathbb{Q}_{\geq0}^\dobs$.
This is done in several steps.
In Theorem~\ref{theorem::innerapx::Risconvexapx}, we show that the set~$R$ is a $(1+\varepsilon)$-convex approximation set.
Then, we prove some auxiliary results to show in Theorem~\ref{theorem::innerapx::runningTime} that Algorithm~\ref{algo::innerapx} runs in polynomial time (under the aforementioned condition to the oracle).
Afterwards, we shortly discuss how to obtain an initial solution.

\begin{theorem}\label{theorem::innerapx::Risconvexapx}
    The set~$R$ computed by Algorithm~\ref{algo::innerapx} is a $(1+\varepsilon)$-convex approximation set.
\end{theorem}
\begin{proof}{Proof}
    At termination, the plane separating oracle has been called with the modified halfspace for every facet of~$\appri$ and returned no solution.
    Thus, for every image~$y$, the point $(1+\varepsilon) y$ is on the positive side of every halfspace that is part of the $\mathcal{H}$-representation of~$\appri$ (see Figure~\ref{fig::innerapx::run}).
    Since $(1+\varepsilon) y\in\appri$ for all $y\in \solob$, it must also hold that $\edgi_{1+\varepsilon}\subseteq\appri$.
    Applying Proposition~\ref{proposition::prelims::edgiInAppri} then shows that~$R$ is a $(1+\varepsilon)$-convex approximation set.
\end{proof}

Next, we prove several auxiliary results that lead to a polynomial running time bound.
To ensure that each single operation outside the oracle calls can be performed in polynomial time, we show that at each point in the algorithm, $\apprii{k}$ is a rational polyhedron where both its vertices and its facets can be encoded with bit length in \oPoly{\langle I \rangle}:

\begin{lemma}\label{lemma::innerapx::polyencoding}
    For each polyhedron $\apprii{k}$, all vertices and all facet-supporting halfspaces have encoding length in~\oPoly{\langle I \rangle}.
\end{lemma}
\begin{proof}{Proof}
    By construction, each~$\apprii{k}$ is a polyhedron in $\mathcal{V}$-representation.
    All vertices of~$\apprii{k}$ are images of solutions returned by the plane separating oracle in previous iterations.
    As such, their encoding length is in \oPoly{\langle I \rangle}.

    Thus, $\apprii{k}$ is a rational polyhedron, and the encoding length of each facet is bounded by a polynomial in the maximum encoding length among the vertices \citep{grotschel.lovasz.ea1988RationalPolyhedra}.
\end{proof}

Due to Lemma~\ref{lemma::innerapx::polyencoding}, the cost of every elementary arithmetic operation is bounded in $\oPoly{\langle I\rangle}$.
For simplicity, we now assume such costs to be in $\bigO{1}$.
This allows us to state a running time for Algorithm~\ref{algo::innerapx} that only depends on the number of solutions that are returned and the running time of the plane separating oracle.

\begin{lemma}\label{lemma::innerapx::outputRunningTime}
    The running time of Algorithm~\ref{algo::innerapx} is in $\mathcal{O}\left(\left|R\right|^{\left\lfloor\frac{d}{2}\right\rfloor+1}\left(\left\lfloor\frac{d}{2}\right\rfloor\log\left|R\right|+ T_{\hpO}\right)\right)$, where~$\left|R\right|$ denotes the cardinality of the set~$R$ at the time of termination and~$T_{\hpO}$ is an upper bound on the running time of~$\hpO$.
\end{lemma}
\begin{proof}{Proof}
    There are $\left|R\right|$ different inner approximation polyhedra throughout the execution of Algorithm~\ref{algo::innerapx}, each with at most $\left|R\right|$~vertices.

    For each polyhedron~$\apprii{k}$, Algorithm~\ref{algo::innerapx} enumerates the facets, filters out halfspaces that have been checked in previous iterations, and calls the plane separating oracle iteratively with the halfspaces induced by the remaining facets.
    Enumerating the facets can be done with the facet enumeration from \citet{chazelle1993OptimalConvexHull} in $\mathcal{O}\left(\left|R\right|\log\left|R\right|+\left|R\right|^{\left\lfloor\frac{d}{2}\right\rfloor}\right)$.
    The number of facets  can be bounded by the upper bound theorem \citep{seidel1995UpperBoundTheorem} and is in $\mathcal{O}\left({\left|R\right|}^{\left\lfloor\frac{d}{2}\right\rfloor}\right)$.
    By organizing the set~$F$ appropriately, looking up whether a halfspace has been checked in a previous iteration incurs an overhead of at most $\bigO{\log \left| F\right|}\in\bigO{\log\left(\left| R\right|^{2\cdot\left\lfloor\frac{d}{2}\right\rfloor}\right)}$.
    The same overhead occurs each time a new halfspace is added to $F$.

    \noindent
    Overall, this leads to a running time bound in
    \[
        \mathcal{O}\left(\left|R\right|\left(\left|R\right|\log\left|R\right|+\left|R\right|^{\left\lfloor\frac{d}{2}\right\rfloor}\cdot \left(1+ \left\lfloor\frac{d}{2}\right\rfloor\log\left|R\right|+T_{\hpO}\right)\right)\right).
    \]
    Simplified, this yields the desired term.
\end{proof}

Lemma~\ref{lemma::innerapx::outputRunningTime} implies that if~$\left|R\right|$ is in $\oStand$, the running time of Algorithm~\ref{algo::innerapx} is in $\oPoly{\langle I \rangle,\varepsilon^{-1},T_{\hpO}}$.
To show that this is the case, we first need the following result:

\begin{lemma}\label{lemma::innerapx::HrecessionCone}
    For each halfspace $H=\left\{z\in\mathbb{R}^{\dobs} : w^\T z\geq c\right\}$ that appears in Step~4 of Algorithm~\ref{algo::innerapx}, it holds that $w\in\minCone$.
\end{lemma}
\begin{proof}{Proof}
    Each halfspace~$H$ that appears is valid for the polyhedron $\apprii{k}$ of the current iteration~$k$.
    By construction, $\apprii{k}$ has the recession cone~$\minCone$.
    Due to this recession cone, the linear program $\min_{z\in \minCone} w^\T z$ is bounded, as otherwise the halfspace cannot be valid.
    Then, its dual linear program $\max0\ \mathrm{s.t.}\,0\leq w$ is feasible, which shows that $w\geq0$.
\end{proof}

We now bound the number of solutions that can be returned by the oracle.
Recall the construction of~$\subdi$ from Definition~\ref{def::prelims::grid}.
The values~$\LB$ and~$\UB$ are determined via the bound on the encoding length of points in a set.
By definition, a plane separating oracle returns only solutions whose encoding length is bounded in \oPoly{\langle I \rangle}.
Every image of such a solution also has an encoding length bounded in \oPoly{\langle I \rangle}.
Hence, the grid~$\subdi$ can be constructed with~$\LB$ and~$\UB$ determined by the bound on the encoding length of the images.
For the remainder of this section, $\subdi$ is assumed to be constructed in this way.

\begin{lemma}\label{lemma::innerapx::onceperS}
    Consider Algorithm~\ref{algo::innerapx} with a plane separating oracle~$\hpO$.
    Let~$S$ be any rectangle from~$\subdi$.
    Then there is at most one solution with an image in~$S$ returned by $\hpO$ throughout the entire algorithm.
\end{lemma}
\begin{proof}{Proof}
    Let~$y^*$ be the image of the first solution returned by~$\hpO$ with an image in~$S$ (the case where there is no such~$y^*$ is trivial).
    Let~$\apprii{k'}$ be the inner approximation polyhedron that is constructed after this solution was returned.
    Because $\apprii{k'}\subseteq\apprii{k}$ for $k\geq k'$, it holds that $y^*\in\apprii{k}$ for each $k\geq k'$.
    Let $H=\left\{z\in\mathbb{R}^\dobs : w^\T z\geq c\right\}$ be a facet-supporting halfspace of $\apprii{k}$ with $k\geq k'$.

    Let~$\ideal{s}$ denote the minimal vertex of~$S$.
    From $w\in \minCone$ (Lemma~\ref{lemma::innerapx::HrecessionCone}) and $y^*\in S$, it follows that $w^\T y^* \leq (1+\varepsilon)w^\T \ideal{s}$ (Lemma~\ref{lemma::grid::alphaprops}).
    As~$y^*$ also is contained in~$H$, this leads to $c\leq w^\T y^* \leq (1+\varepsilon)w^\T \ideal{s}$.

    Now let~$y$ denote the image of a solution returned by~$\hpO$ if it was called for the halfspace $H_{1+\varepsilon}=\left\{z\in\mathbb{R}^\dobs : (1+\varepsilon)w^\T z\geq c\right\}$ and returned a solution.
    As it was returned by~$\hpO$, $y$ is not in $H_{1+\varepsilon}$ and, therefore, $(1+\varepsilon)w^\T y<c\leq(1+\varepsilon)w^\T \ideal{s}$.
    Furthermore, $(1+\varepsilon)w$ is in $\minCone$ and, thus, $(1+\varepsilon )w^\T y<(1+\varepsilon)w^\T s$ for all $s\in S$ (Lemma~\ref{lemma::grid::alphaprops}).
    Hence, $y\notin S$.
\end{proof}

\noindent
This leads to the central theorem on the running time:

\begin{theorem}\label{theorem::innerapx::runningTime}
    Algorithm~\ref{algo::innerapx} has a running time in $\oPoly{\langle I \rangle, \varepsilon^{-1}, T_{\hpO}}$, where~$T_{\hpO}$ is an upper bound on the running time of the plane separating oracle.
\end{theorem}
\begin{proof}{Proof}
    By Lemma~\ref{lemma::innerapx::onceperS}, at most $\left|\subdi\right|$~solutions are added to $R$, and $\left|\subdi\right|$ is in \oStand.
    Plugging this into Lemma~\ref{lemma::innerapx::outputRunningTime} gives a running time bound in $\oPoly{\langle I \rangle, \varepsilon^{-1}, T_{\hpO}}$.
\end{proof}

Concerning the running time bound~$T_{\hpO}$ for the plane separating oracle, $T_{\hpO}$ is in $\oPoly{\langle I\rangle}$ if the plane separating oracle has a running time in $\oPoly{\langle H, I\rangle}$ for a halfspace $H$.
Due to Lemma~\ref{lemma::innerapx::polyencoding}, $\langle H \rangle\in\oPoly{\langle I\rangle}$ holds for every halfspace given to $\hpO$ in Algorithm~\ref{algo::innerapx}.

\smallskip

To use Algorithm~\ref{algo::innerapx} as a $(1+\varepsilon)$-convex approximation algorithm, an initial solution must first be obtained in polynomial time.
In theory, this can be achieved by using the plane separating oracle in a binary search scheme.
In practice, much faster heuristics are often at hand.

After termination, the set~$R$ is a $(1+\varepsilon)$-convex approximation set, but may still contain solutions that do not map to vertices of~$\appri$ (as seen in Figure~\ref{fig::innerapx::run}).
These solutions are redundant for a $(1+\varepsilon)$-convex approximation set and can be filtered by a post-processing step in polynomial time.

\smallskip

Although not needed to prove polynomial running time, the bound on the number of returned solutions can be tightened a bit if the stronger plane separating oracles from \citet{csirmaz2021InnerApproximationAlgorithm} are used.
There, the oracle must return solutions that are supported or even extreme supported.
Combining Lemma~\ref{lemma::innerapx::onceperS} with the bound on hyperrectangles of $\subdi$ containing supported solutions then shows that the total number of solutions returned by such a stronger oracle in Algorithm~\ref{algo::innerapx} is bounded in $\bigO{{(\varepsilon^{-1} p)}^{\dobs-1}}$,
where~$p$ is the bound on the encoding length of returned images.

Furthermore, a basic property of the exact inner approximation algorithm from \citet{csirmaz2021InnerApproximationAlgorithm} carries over when these stronger oracles are used.
If a plane separating oracle returns only supported solutions, each with an image maximizing the distance to the respective halfspace, at most $\left|\face(\edgi)\right|$ solutions can be returned in Algorithm~\ref{algo::innerapx}.
If it returns only extreme supported solutions, at most $\left|\vertx(\edgi)\right|$ solutions can be returned.
While the second result is self-explanatory, the proof for the first result given in \citet{csirmaz2021InnerApproximationAlgorithm} also applies when the halfspace~ $H_{1+\varepsilon}$ is considered instead of~$H$.
Hence, when a stronger plane separating oracles is used, the running time of Algorithm~\ref{algo::innerapx} is bounded by~$\left|\subdi\right|$ and the structural complexity of~$\edgi$.
This ensures that Algorithm~\ref{algo::innerapx} can adapt to structurally easy instances, even if~$\left|\subdi\right|$ is rather high.

\section{The Weighted Sum Scalarization as Oracle}\label{sec::innerapx::WS}

In this section, we construct practical plane separating oracles.
A well-known scalarization for MOMILPs is the \emph{weighted sum scalarization} $\WS(w)$ \citep{bazgan.ruzika.ea2022PowerWeightedSum}.
It is parametrized by a weight~$w\in \minCone$ and is the single-objective optimization problem
\[
\begin{array}{l l}
    \min & w^\T f(x) \\
    \mathrm{s.t.} & x\in\solfeas\text{.}
\end{array}
\]

It is well-known that every optimal solution to a weighted sum scalarization is supported \citep{ehrgott2005MulticriteriaOptimization}.
Moreover, for every weighted sum scalarization, at least one optimal solution has an image that is also a vertex of~$\edgi$.
Thus, it can be assumed that every solution returned by a weighted sum oracle has encoding length in \oPoly{\langle I \rangle}.

In \citet{csirmaz2021InnerApproximationAlgorithm}, the weighted sum scalarization is used to construct a plane separating oracle.
This oracle satisfies the stronger conditions outlined at the end of Section~\ref{sec::innerapx} --- it returns only supported solutions with maximal distance to the halfspace.
Algorithm~\ref{algo::innerapx::wsexact} shows the construction of the oracle.
With some abuse of notation, we denote by $\WS(w)$ an algorithm that solves the weighted sum scalarization.

\input{figures/hpoWS.tex}

\begin{corollary}
 Let $\WS(w)$ be solvable in time polynomial in $\oPoly{\langle I,w\rangle}$ for each $w\in\minCone$.
 Then, Algorithm~\ref{algo::innerapx} can be used to construct a $\cfptas$ by using $\hpOWS$ as the plane separating oracle.
\end{corollary}
Some selected applications where $\hpOWS$ can be solved in polynomial time are listed (among others) in Table~\ref{table::innerapx::problems}. However, there are many problems for which the existence of a polynomial-time weighted sum algorithm (or any other polynomial-time plane separating oracle) would imply $\P=\NP$, for example the multi-objective variants of the travelling salesman problem or the knapsack problem.
Yet, for some of these problems, single-objective approximation algorithms can be used to approximate the weighted sum scalarization.
\begin{definition}\label{def::innerapx::apxws}
    For a minimization MOMILP and some $\alpha\in\mathbb{R}_{>1}$, let $\WS_\alpha$ be an $\alpha$-approximation algorithm for the weighted sum scalarization.
    In Step~2 of Algorithm~\ref{algo::innerapx::wsexact}, use $\WS_\alpha(w)$ instead of $\WS(w)$.
    The resulting oracle is denoted by $\hpOWS_\alpha$.
\end{definition}
Again, we assume that each solution returned by $\hpOWS_\alpha$ has an encoding length in \oPoly{\langle I \rangle}.

If Algorithm~\ref{algo::innerapx} is executed with $\hpOWS_\alpha$ as the plane separating oracle, Lemmas~\ref{lemma::innerapx::polyencoding},~\ref{lemma::innerapx::outputRunningTime} and~\ref{lemma::innerapx::onceperS} directly apply as they only rely on the polynomial encoding length of the returned solutions.
Hence, the running time result from Theorem~\ref{theorem::innerapx::runningTime} still holds.
However, as the returned solutions are not necessarily supported, the tighter bounds from the end of Section~\ref{sec::innerapx} do not apply anymore.

\noindent
With the alternative oracle, the approximation factor of Algorithm~\ref{algo::innerapx} changes.
\begin{theorem}\label{theorem::innerapx::apxRisconvexapx}
    The set~$R$ computed by Algorithm~\ref{algo::innerapx} when using $\hpOWS_\alpha$ as plane separating oracle is an $\alpha\cdot(1+\varepsilon)$-convex approximation set.
\end{theorem}
\begin{proof}{Proof}
    Let $H=\left\{z\in\mathbb{R}^\dobs : w^\T z\geq c\right\}$ be any facet-supporting halfspace of $\appri$.
    After termination of Algorithm~\ref{algo::innerapx}, $\hpOWS_\alpha$ must have been called for $H_{1+\varepsilon}=\left\{z\in\mathbb{R}^\dobs : (1+\varepsilon)w^\T z\geq c\right\}$ without returning a solution.
    In this call, the oracle invokes $\WS_\alpha((1+\varepsilon)w)$, which must have returned a solution~$x^*$.
    Since $\WS_\alpha$ is an $\alpha$-approximation algorithm, $(1+\varepsilon)w^\T f(x^*)\leq \alpha(1+\varepsilon)w^\T f(x)$ holds for all~$x\in X$.
    Furthermore, since the oracle~$\hpOWS_\alpha$ did not return~$x^*$, $(1+\varepsilon)w^\T f(x^*)\geq c$ must also hold.
    Hence, $ w^\T(\alpha(1+\varepsilon)f(x))\geq c$ holds for all $x\in \solfeas$.
    Because this applies to all facet-supporting halfspaces of~$\appri$, it follows that $\alpha(1+\varepsilon)f(x)\in\appri$.
    Then, Proposition~\ref{proposition::prelims::edgiInAppri} shows that~$R$ is an $\alpha(1+\varepsilon)$-convex approximation set.
\end{proof}

\noindent
If $\WS_\alpha$ is a PTAS or an FPTAS, some stronger results are achievable.
\begin{theorem}\label{theorem::innerapx::cfptas}
    If $\WS_\alpha$ is an (F)PTAS, Algorithm~\ref{algo::innerapx} can be used to construct a \cbothptas.
\end{theorem}
\begin{proof}{Proof}
    Let a MOMILP instance~$I$, $\varepsilon\in\mathbb{Q}_{>0}$, and an (F)PTAS $\WS_{\alpha}$ for~$I$ be given.
    A $(1+\varepsilon)$-convex approximation algorithm can be constructed in two steps:
    First, $\beta\in\mathbb{Q}_{>0}$ and~$\gamma\in\mathbb{Q}_{>0}$ are chosen such that $(1+\beta)(1+\gamma)\leq1+\varepsilon$.
    Furthermore, the values $\beta^{-1}$ and $\gamma^{-1}$ must be bounded in $\oPoly{\varepsilon^{-1}}$ and the encoding lengths of~$\beta$ and~$\gamma$ must be bounded in $\oPoly{\langle\varepsilon\rangle}$.
    Algorithm~\ref{algo::innerapx} is then invoked with $\hpOWS_{1+\beta}$ as the oracle and~$(1+\gamma)$ instead of $(1+\varepsilon)$ as approximation factor.

    By Theorem~\ref{theorem::innerapx::apxRisconvexapx}, this procedure returns a $(1+\beta)(1+\gamma)$-convex approximation set, which then also is a $(1+\varepsilon)$-convex approximation set.
    By Theorem~\ref{theorem::innerapx::runningTime}, the procedure has the running time properties of a \cbothptas.
\end{proof}

An example for choosing $\beta$ and $\gamma$ so that they satisfy the conditions from the proof of Theorem~\ref{theorem::innerapx::cfptas} is to set $\beta\coloneq\nicefrac{\varepsilon}{2}$ and $\gamma\coloneq\nicefrac{1+\varepsilon}{1+\beta}-1$.
In Table~\ref{table::innerapx::problems}, we list some selected applications where an~$\alpha$ exists so that $\hpOWS_\alpha$ can be solved in polynomial time.

\begin{table}
    \centering
    \caption{Selected applications and the best achievable approximation factor for convex approximation\label{table::innerapx::problems}}
    {\begin{tabular}{l l l l}
        \toprule
        problem & $\alpha$ & $\mathcal{O}$  & convex-apx \\
        \midrule
        bounded MOLP & $1.0$ & $\mathsf{poly}\langle I \rangle$ & \cfptas \\
        MO single-pair shortest path & $1.0$ & $m + n \log n $ & \cfptas \\
        MO assignment & $1.0$ & $n^3$ & \cfptas \\
        MO knapsack & FPTAS & $n \log n + \frac{\ln(\epsilon^{-1})}{\epsilon^{4}} $ & \cfptas \\
        MO travelling salesman & $\log n$ & $n^3$ &  $\log n \cdot (1+\varepsilon)$ \\
        MO metric travelling salesman & $1.5$ & $n^3$ & $1.5\cdot(1+\varepsilon)$  \\
        \bottomrule
    \end{tabular}}

    \smallskip
    
    \begin{minipage}{0.84\textwidth}
        \scriptsize
        The column~$\alpha$ denotes the best possible approximation factor for the weighted sum scalarization, the column~$\mathcal{O}$ denotes the asymptotic running time of a corresponding approximation algorithm, and the column convex-apx denotes the resulting convex approximation factor for Algorithm~\ref{algo::innerapx}.
        The running times and approximation factors are either well-known or described in \citet{vazirani2001ApproximationAlgorithms}.
        The abbreviation \textit{MO} stands for \textit{multi-objective}.
    \end{minipage}
\end{table}

\section{Computational Results}

To evaluate the performance of Algorithm~\ref{algo::innerapx} with $\hpOWS$ as the oracle, we conduct a numerical study.
Our study consists of two parts.
In the first part, we follow the study conducted in~\citet{helfrich.ruzika.ea2024EfficientlyConstructing} by using the same metrics and instances, supplemented by an additional instance set.
In the second part, we evaluate the performance of our algorithm for metrics that measure the quality of solution sets as so-called \emph{representations} by following the setup developed in~\citet{sayin2024SupportedNondominatedPoints}. %

For the remainder of this section, \iaa{} denotes our algorithm.
A key result from \citet{helfrich.ruzika.ea2024EfficientlyConstructing} is that the algorithm therein runs much faster than other existing convex approximation algorithms.
Therefore, we use it as the benchmark algorithm for our comparisons.
The algorithm from \citet{helfrich.ruzika.ea2024EfficientlyConstructing} is denoted by \oaa{}.
It should be noted that, in theory, the worst case running time of \oaa{} is better than that of \iaa{}.
We omit a comparison to exact algorithms that compute the entire non-dominated set.
Computational studies (e.g., in~\citet{daechert.fleuren2024SimpleEfficientVersatile}) show that such algorithms already take hours on instances for which \iaa{} and \oaa{} terminate in seconds.%

\smallskip

Both \iaa{} and \oaa{} are implemented in C++17 and compiled with gcc 12.2.0.
We provide a code and a cmake build script in our git repository.\footnote{\href{https://github.com/INFORMSJoC/2025.1308}{\label{footnote1}\texttt{https://github.com/INFORMSJoC/2025.1308}}}
The facet enumeration needed for \iaa{} is implemented as a vertex enumeration for a dual polyhedron of~$\appri$ \citep{heyde.lohne2008GeometricDualityMultiple}.
The cdd library \citep{fukuda1997cdd} is used for the vertex enumerations of both \iaa{} and \oaa{}.
The computations were run on an Intel Xeon Gold 6226R CPU with 2.90GHz, 450GB of RAM, and Debian 6.1.\ as the operating system.
Each instance was run with a running time limit of one hour.

\smallskip

We start with the first part of our study.
For the problem instances, we combine three well-established sets of instances.
The first two are the sets that are also used in \citet{helfrich.ruzika.ea2024EfficientlyConstructing}.
They consist of instances of the $3$-objective knapsack problem (KP) and instances of the $3$-objective symmetric metric travelling salesman problem (TSP).
Weighted sum scalarizations for KP are approximated by a factor of~$0.5$ with the extended greedy algorithm \citep{kellerer.pferschy.ea2004KnapsackProblems} and for TSP by a factor of~$1.5$ with the algorithm from \citet{christofides2022WorstCaseAnalysisNew}.
To implement the latter,  we use the ogdf library \citep{chimani2013open}.
The third set of instances is a sample from the set of multi-objective assignment problem (AP) instances generated for the computational study in \citet{bokler.mutzel2015OutputSensitiveAlgorithmsEnumerating}.
This set contains problems with three to six objectives.
Weighted sum scalarizations for AP can be solved exactly in polynomial time, we use Gurobi~12 for this.

Both the AP and KP instances were generated by the same scheme as the often-used benchmark instances from~\citet{kirlik.sayin2014NewAlgorithmGenerating}.
However, convex approximation algorithms terminate in seconds on these benchmark instances, which is why we use the larger instances from~\citet{bokler.mutzel2015OutputSensitiveAlgorithmsEnumerating} and~\citet{helfrich.ruzika.ea2024EfficientlyConstructing}.
A summary of the instance set for the first part is given in Table~\ref{table::study::instances}.
For the complete set, see our git repository.\textsuperscript{\ref{footnote1}}

\begin{table}
\centering
    \caption{The instances used in the first part of our computational study.\label{table::study::instances}}
    {\begin{tabular}{l r r r r r l}
        \toprule
        type & $\dobs$ & sizes & objectives & $\#$ & $\alpha$ & from \\
        \midrule
        KP & $3$ & $10,20,\ldots,250$ & uniform& $5$ & $0.5$ & \citet{helfrich.ruzika.ea2024EfficientlyConstructing} \\
        & $3$ & $10,20,\ldots,250$ & conflicting& $5$ & $0.5$ & \citet{helfrich.ruzika.ea2024EfficientlyConstructing} \\
        TSP & $3$ & $10,20,\ldots,250$ & uniform & $5$ & $1.5$ & \citet{helfrich.ruzika.ea2024EfficientlyConstructing} \\
        AP & $3$ & $40,60,\ldots,200$ & uniform & $10$ & $1.0$ & \citet{bokler.mutzel2015OutputSensitiveAlgorithmsEnumerating} \\
         & $4$ & $10,20,\ldots,100$ & uniform & $10$ & $1.0$ & \citet{bokler.mutzel2015OutputSensitiveAlgorithmsEnumerating} \\
         & $5$ & $8,10,\ldots,22$ & uniform & $10$ & $1.0$ & \citet{bokler.mutzel2015OutputSensitiveAlgorithmsEnumerating} \\
         & $6$ & $4,6,\ldots,22$ & uniform & $10$ & $1.0$ & \citet{bokler.mutzel2015OutputSensitiveAlgorithmsEnumerating} \\
        \bottomrule
    \end{tabular}}

    \smallskip
    
    \begin{minipage}{0.85\textwidth}
    \scriptsize
        The column $\#$ denotes the number of instances of each size and the column $\alpha$ states the approximation factor for the weighted sum scalarization.
    \end{minipage}
    {}
\end{table}

To evaluate the performance of the algorithms in the first part of our study, we use the same three metrics as \citet{helfrich.ruzika.ea2024EfficientlyConstructing}.
The first metric is the total running time of each algorithm in seconds.
The second metric is the $\varepsilon$-convex indicator.
It measures the minimum value of $1+\varepsilon$ for which a solution set $R$ is still a $(1+\varepsilon)$-convex approximation set.
The third metric measures the ratio $\left|R\right|/\left|R^*\right|$, where~$R$ denotes the output set of a convex approximation algorithm and~$R^*$ the output set of the dual Benson algorithm.
Note that its output set might contain some non-extreme supported solutions.
In practice, the number of such is often negligible \citep{bokler.mutzel2015OutputSensitiveAlgorithmsEnumerating} and can be ignored here.
The dual Benson algorithm is also subject to the time limit of one hour.
Since we need its outcome set to compute both the $\varepsilon$-convex indicator and $\left|R\right|/\left|R^*\right|$, we can provide these metrics only for the smaller instances where it terminates in time.%

\begin{figure}[ht!]
    \caption{Results for the TSP instances\label{fig::study::tsp}}

    \centering
    \includegraphics[width=1\textwidth]{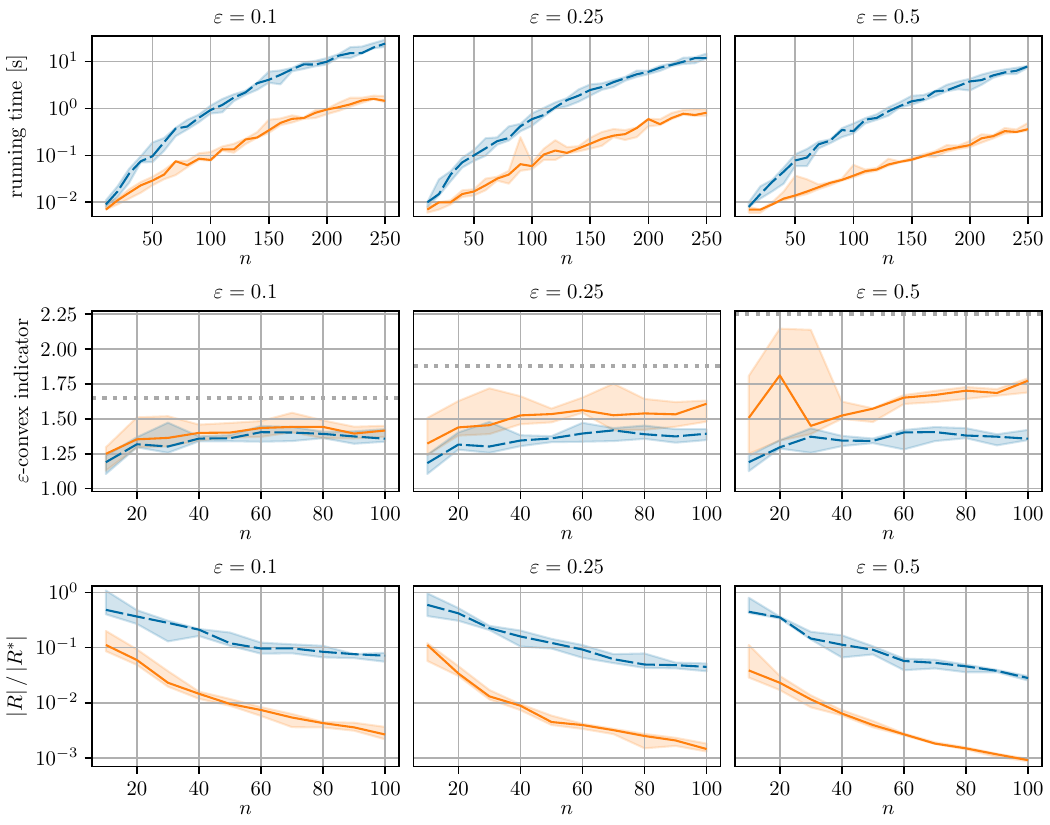}

    \begin{minipage}{0.95\textwidth}
        \footnotesize{%
            The value of~$n$ denotes the number of cities.
            Solid lines are for \iaa{}, dashed lines for \oaa{}.
            The lines always show the median value, shaded areas show the complete range of values.
            Dotted lines in the second row mark the maximum possible $\varepsilon$-convex indicator~$\alpha\cdot(1+\varepsilon)$.}
    \end{minipage}
\end{figure}

\begin{figure}[ht!]
    \caption{Results for the $3$-objective AP instances\label{fig::study::ap}}

    \centering
    \includegraphics[width=1\textwidth]{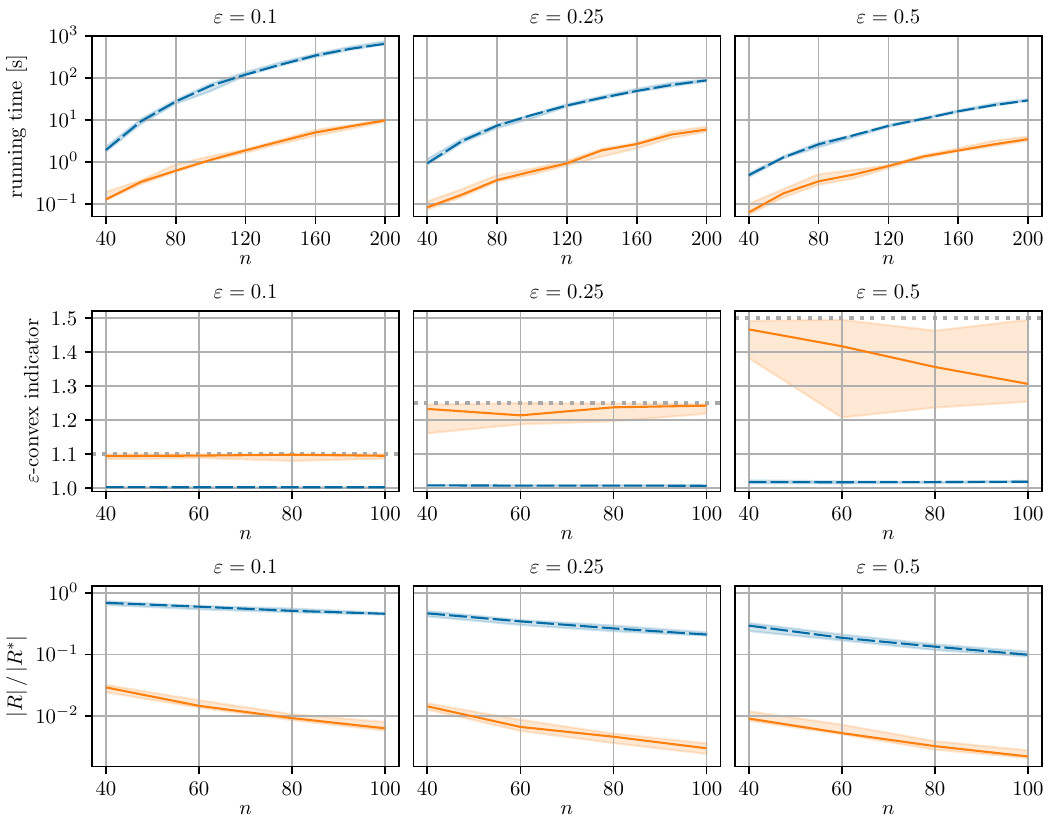}

    \begin{minipage}{0.95\textwidth}
        \footnotesize{%
            The value of~$n$ denotes the number of pairs in an assignment.
            Solid lines are for \iaa{}, dashed lines for \oaa{}.
            The lines always show the median value, shaded areas show the complete range of values.
            Dotted lines in the second row mark the maximum possible $\varepsilon$-convex indicator~$\alpha\cdot(1+\varepsilon)$.}
    \end{minipage}
\end{figure}

\begin{figure}[ht!]
    \caption{Running time results for the AP instances with more than three objectives\label{fig::study::apHigher}}

    \centering
    \includegraphics[width=1\textwidth]{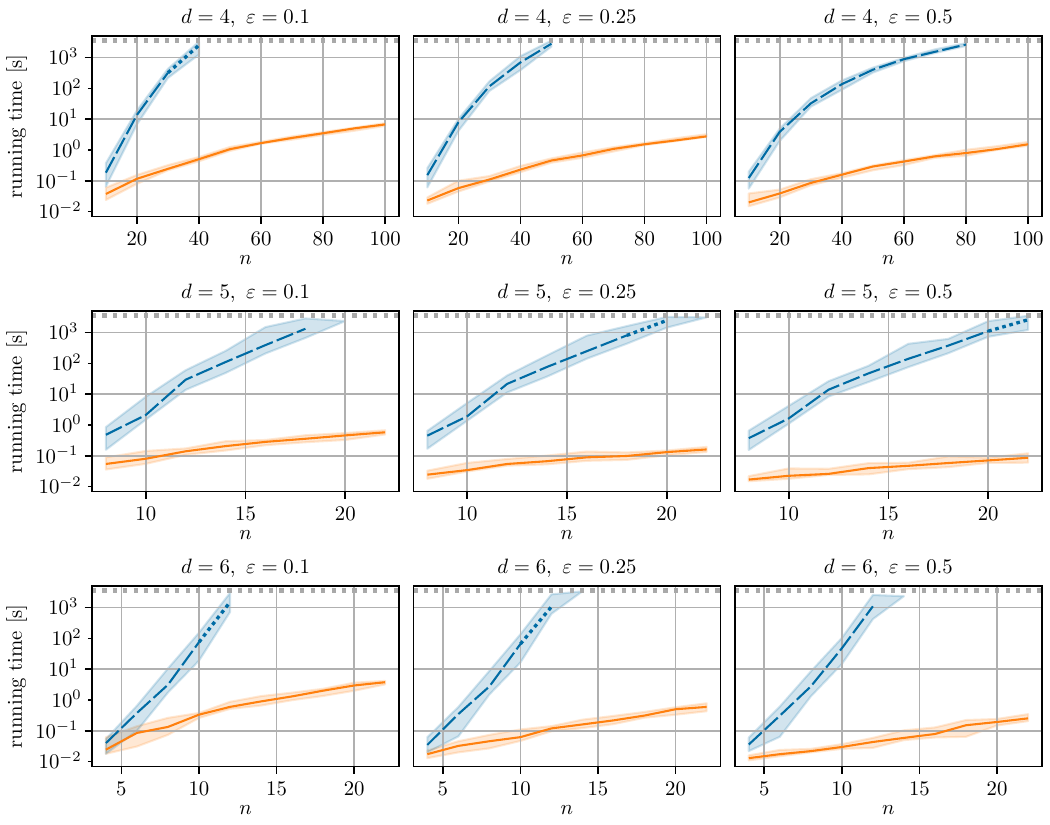}

    \begin{minipage}{0.95\textwidth}
        \footnotesize{%
            The value of~$n$ denotes the number of pairs in an assignment.
            The dotted line at the top marks the time limit of one hour.
            Solid lines are for \iaa{}, dashed lines for \oaa{}.
            The lines always denote the median value, the shaded areas the range of values of instances that finished within the time limit.
            Wherever the \oaa{} line is dotted, not all instances finished within time limit.
            No line indicates that less than half of the instances finished within the time limit.}
    \end{minipage}
\end{figure}

\smallskip

We plot all three metrics for TSP in Figure~\ref{fig::study::tsp} and for $3$-objective AP in Figure~\ref{fig::study::ap}.
For the KP instances, we omit the plots because \iaa{} behaves very similar on all these instances.
Over all KP instances, the cardinality of~$R$ from \iaa{} is at most~$5$ for $\varepsilon=0.1$, at most~$3$ for $\varepsilon=0.25$, and always~$1$ for $\varepsilon=0.5$.
In contrast, the number of solutions that \oaa{} finds grows with the size of the instances.
For KP instances with less than $100$~items, both algorithms terminate effectively immediately, and for instances with $100$ or more items, \iaa{} always terminates faster than \oaa{}.
For the other instance types, the number of solutions output by \iaa{} increases with the instance size, albeit far slower than it does with \oaa{}.

Over all instances classes, a general pattern emerges:
\iaa{} runs much faster and finds far fewer solutions than \oaa{}.
However, at the same time, the $\varepsilon$-convex indicator for \iaa{} is larger than the $\varepsilon$-convex indicator for \oaa{}.
This can be seen most clearly for the $3$-objective AP instances in Figure~\ref{fig::study::ap}.
For these instances, the weighted sum scalarizations are solved exactly.
The $\varepsilon$-convex indicator for \oaa{} is then always close to~$1.0$, even if the required approximation factor is far larger.
In contrast, the $\varepsilon$-convex indicator for \iaa{} is rather close to the required approximation factor.

\smallskip

In general, the findings suggest that the margin allowed by the approximation factor is used more effectively by \iaa{} than by \oaa{}.
This allows \iaa{} to find convex approximation sets faster than \oaa{}, at the cost of a worse $\varepsilon$-convex indicator.
For a (convex) approximation algorithm, this is the desired behavior:
We are willing to sacrifice accuracy to obtain solution sets with faster running times.

\smallskip

For AP instances with more than three objectives, Figure~\ref{fig::study::apHigher} shows that \iaa{} scales a lot better with the number of objectives than \oaa{}.
Often, \oaa{} runs into the time limit of one hour.
The other two metrics show similar results as in the $3$-objective case, so we refrain from plotting them.

\smallskip

In the second part of our study, we consider \iaa{} and \oaa{} for finding \emph{representations}.
A representation is a low-cardinality set of solutions such that its images represent the entire non-dominated set with sufficient \emph{coverage} and \emph{uniformity}.
For an in-depth introduction, we refer to~\citet{sayin2000MeasuringQualityDiscrete}.
Although similar in concept, the notion of (convex) approximation has a different purpose:
It is used to develop polynomial-time algorithms with a provable guarantee on the error.
Nonetheless, we are interested in the quality of the outcome sets of \iaa{} and \oaa{} as representations.%

This part of our computational study follows the setup developed in~\citet{sayin2024SupportedNondominatedPoints}, where the quality of the non-dominated extreme points as representations is measured using four metrics: \emph{coverage error} (CE), \emph{median error} (ME), \emph{hypervolume ratio} (HVR), and \emph{range ration} (RR).
All four metrics are discussed in-depth in~\citet{sayin2024SupportedNondominatedPoints}, and we provide formal definitions in Appendix~\ref{appendix::sayinmetrics}.
As instances, we use the 3-objective AP and KP instances from~\citet{kirlik.sayin2014NewAlgorithmGenerating} since the non-dominated sets for these instances are known (see \href{https://github.com/aritrasep/Modolib.jl/}{\texttt{github.com/aritrasep/Modolib.jl}}).%

Selected results are listed in Table~\ref{table::study::metrics}.
For $\varepsilon=0.1$, \oaa{} performs better than \iaa{} on all instances.
This is hardly surprising:
As seen in the first part of our study, \oaa{} overestimates the approximation error and finds more solutions than \iaa{}.
For all four metrics, a higher number of solutions improves the quality.
However, if~$\varepsilon$ is lowered to $0.01$, \iaa{} still runs faster than \oaa{} with $\varepsilon=0.1$, and often achieves comparable results in its representation quality.
This indicates that \oaa{} computes better representations for a fixed value of~$\varepsilon$, while \iaa{} allows better control on the trade-off between running time and solution quality, for both approximation and representation.%

\smallskip

We conclude with a remark on the results from the study in \citet{helfrich.ruzika.ea2024EfficientlyConstructing}.
While performing very similar on a qualitative level, our implementation of \oaa{} finds somewhat fewer solutions than the implementation from \citet{helfrich.ruzika.ea2024EfficientlyConstructing}.
The main reasons seem to be differently performing vertex enumerations.
We use the cdd library, while qhull \citep{barber.dobkin.ea1996QuickhullAlgorithm} is used in \citet{helfrich.ruzika.ea2024EfficientlyConstructing}.
Most practical implementations of vertex enumerations that use fixed precision arithmetics have to deal with numerical inconsistencies (see, e.g., Section~1 in \citet{lohne2023ApproximateVertex}).
Depending on how an implementation handles such inconsistencies, the set of computed vertices might differ from the set computed by another implementation.
We still opt to use cdd with fixed precision arithmetics, for two reasons:
First, fixed precision arithmetics are much faster than exact arithmetics in practice.
Second, the empirical evidence from our study suggests that both \iaa{} and \oaa{} work correctly as convex approximation algorithms when using cdd.

\input{figures/table_representations.tex}

\section{Conclusion}\label{sec::conclusion}

\enlargethispage{\baselineskip}
This paper presents a new convex approximation algorithm for MOMILPs\@.
It works in a grid-agnostic fashion, which distinguishes it from all previous general convex approximation algorithms.
The only requirement is that a plane separating oracle is at hand.
In the case that no such oracle is available, Algorithm~\ref{algo::innerapx} can use approximation algorithms for weighted sum scalarizations to construct a similar oracle instead.
On a practical set of problem instances, Algorithm~\ref{algo::innerapx} consistently runs faster than the state-of-the art convex approximation algorithm from \citet{helfrich.ruzika.ea2024EfficientlyConstructing}.

In addition to multi-objective optimization, Algorithm~\ref{algo::innerapx} is also relevant for parametric optimization.
Every convex approximation algorithm is likewise a parametric approximation algorithm, and vice versa~\citep{helfrich.herzel.ea2022ApproximationAlgorithmGeneral}.
Notably, this correspondence extends even to non-convex and combinatorial parametric optimization problems.
Hence, Algorithm~\ref{algo::innerapx} extends the toolbox for practitioners in the field of parametric optimization.

In~\citet{csirmaz2021InnerApproximationAlgorithm}, an outer approximation algorithm is developed alongside the inner approximation algorithm.
We remark that this outer approximation algorithm can be adapted similarly to the inner approximation algorithm so that it runs in polynomial time.
However, constructing practical oracles for the outer approximation algorithm appears to be more complicated than for the inner approximation algorithm (see \citet{bokler.parragh.ea2024OuterApproximationAlgorithm}).
Furthermore, computational studies in \citet{hamel.lohne.ea2014BensonTypeAlgorithms} and \citet{bokler.parragh.ea2024OuterApproximationAlgorithm} indicate that, for many practical instances, the inner approximation algorithm runs faster than the outer approximation algorithm.
For this reason, we omit a discussion of the outer approximation algorithm.

\AtNextBibliography{\small\singlespacing}
\setlength\bibitemsep{0.3\baselineskip}
\printbibliography[title={References},heading={bibintoc}]

\input{appendix.tex}

\end{document}

%% file: figures/innerapx.tex
\begin{algorithm}
    \caption{Inner approximation algorithm for $(1+\varepsilon)$-convex approximation, derived from the exact inner approximation algorithm in~\cite{csirmaz2021InnerApproximationAlgorithm}.
    The oracle $\hpO$ and the value $\varepsilon$ parameterize the algorithm and are not considered part of the input.}\label{algo::innerapx}
    \begin{algorithmic}[1]
    \Require{}MOMILP instance~$I$ and an arbitrary feasible solution $x^\mathrm{init}\in\solfeas$ (with encoding length in \oPoly{\langle I \rangle})
    \Ensure{}$R$ is a $(1+\varepsilon)$-convex approximation set and $\edgi_{1+\varepsilon}\subseteq\appri\subseteq\edgi$
    \State{} $R \gets \left\{x^\mathrm{init}\right\}$, $k \gets 1$, $F \gets \emptyset$
    \State{} $\apprii{1}\gets \left\{f(x^\mathrm{init})\right\}+\minCone$
    \While{$\halfsp\left(\apprii{k}\right)\setminus F\neq \emptyset$}\label{algo::innerapx::loopstart}
    \State{} $H\gets$ arbitrary halfspace $\{z\in\mathbb{R}^\dobs: w^\T z\geq c\}\in\halfsp\left(\apprii{k}\right)\setminus F$
    \State{} $H_{1+\varepsilon}\gets\{z\in\mathbb{R}^d: (1+\varepsilon)w^\T z\geq c\}$
    \State{} $\left(\textit{status},x^*\right)\gets\hpO(H_{1+\varepsilon})$
    \If{$\textit{status}=\textit{not inside}$}
        \State{} $R\gets R\cup\{x^*\}$, $k\gets k+1$
        \State{} $\apprii{k}\gets\conv\left\{f(x) : x\in R\right\}+\minCone$
    \Else{}
        \State{} $F\gets F\cup\{H\}$
    \EndIf{}
    \EndWhile{}
    \State{} $\appri\gets\apprii{k}$
    \end{algorithmic}
\end{algorithm}

%% file: figures/innerapx-example.tex
\begin{figure}
    \caption{
        Example for a run of Algorithm~\ref{algo::innerapx}.
    }\label{fig::innerapx::run}
    \centering
    \begin{minipage}[t]{0.31\textwidth}
        \centering
        \begin{tikzpicture}[scale=0.53]

            \def\xbound{7.3}
            \def\ybound{6.9}
            \def\epsshift{1.3}

            \draw[->, thick] (0,0)--(\xbound,0) node[pos=0.5, below]{$f_1$};
            \draw[->, thick] (0,0)--(0,\ybound) node[pos=0.5, left]{$f_2$};

            \yset

            \foreach \i in {1,...,9} {
                \node (yeps\i) at ($(0,0)!\epsshift!(y\i)$){};
            }

            \foreach \i in {1,...,9} {
                \draw[cbred!80, thick, shorten >= 4pt, ->] (y\i.center) to ($(0,0)!\epsshift!(y\i)$.center);
                \draw[draw=none,fill=cbred!80, thick, ->] ($(0,0)!\epsshift!(y\i)$) circle[radius=4pt];

                \draw[fill=black!20] (y\i) circle[radius=5pt];
            }

            \node[black,inner sep=0, below right = 0 pt and -6.5 pt of y3] () {\scriptsize $y$};
            \node[cbred,inner sep=0, below right = -1 pt and -6.5 pt of yeps3] () {\scriptsize $(1+\varepsilon) y$};

        \end{tikzpicture}
        \subcaption*{The set of images of a MOMILP and the position of the images shifted by factor $(1+\varepsilon)$}
    \end{minipage}\hfill
    \begin{minipage}[t]{0.31\textwidth}
        \centering
        \begin{tikzpicture}[scale=0.53]

            \def\xbound{7.3}
            \def\ybound{6.9}
            \def\epsshift{1.3}

            \draw[->, thick] (0,0)--(\xbound,0) node[pos=0.5, below]{$f_1$};
            \draw[->, thick] (0,0)--(0,\ybound) node[pos=0.5, left]{$f_2$};

            \yset

            \node (yupper) at (3.8,\ybound) {};
            \node (yright) at (\xbound,3.5) {};
            \node (ynadir) at (\xbound,\ybound) {};

            \draw[draw=none, fill=cbskyblue!20,] (yupper.center) to (y3.center)to (yright.center) to (ynadir.center) to cycle;

            \draw[draw=cbskyblue!60, thick,arrows = {Stealth[length=10pt, inset=2pt]-Stealth[length=10pt, inset=2pt]}] (yupper.center) to (y3.center) to (yright.center);

            \foreach \i in {1,...,9} {
                \draw[fill=black!20] (y\i) circle[radius=5pt];
                }

            \draw[fill=cbskyblue!100] (y3) circle[radius=5pt];

            \node[anchor=south west,color=cbdarkblue!]() at (5.5,4.5) {$\apprii{1}$};

            \node[black,inner sep=0, below right = -1 pt and -8 pt of y3] () {\scriptsize $f(x^\mathrm{init})$};

        \end{tikzpicture}
        \subcaption*{$\apprii{1}$ is constructed from the image of the initial solution}
    \end{minipage}\hfill
    \begin{minipage}[t]{0.31\textwidth}
        \centering
        \begin{tikzpicture}[scale=0.53]

            \def\xbound{7.3}
            \def\ybound{6.9}
            \def\epsshift{1.3}

            \draw[->, thick] (0,0)--(\xbound,0) node[pos=0.5, below]{$f_1$};
            \draw[->, thick] (0,0)--(0,\ybound) node[pos=0.5, left]{$f_2$};

            \yset

            \node (yupper) at (3.8,\ybound) {};
            \node (yright) at (\xbound,3.5) {};
            \node (ynadir) at (\xbound,\ybound) {};

            \draw[draw=none, fill=cbskyblue!20,] (yupper.center) to (y3.center)to (yright.center) to (ynadir.center) to cycle;

            \draw[draw=cbskyblue!60, thick,arrows = {Stealth[length=10pt, inset=2pt]-Stealth[length=10pt, inset=2pt]}] (yupper.center) to (y3.center) to (yright.center);

            \foreach \i in {1,...,9} {
                \draw[fill=black!20] (y\i) circle[radius=5pt];
            }

            \draw[draw=cbgreen!100, dashed, ultra thick] ($(yupper)!0.14!(y3)$) to ($(yupper)!2.0!(y3)$);

            \node (mid) at ($(yupper)!1.35!(y3)$) {};
            \node[cbgreen!100,inner sep=0, above right = 0 pt and -2 pt of mid] () {$H$};

            \draw[draw=cbgreen!100, thick, arrows=-{Latex[width=6pt,length=8pt]}] (mid.center) to ($(mid)!1.0! -90:(y3)$);

            \draw[fill=cbskyblue!100] (y3) circle[radius=5pt];

            \draw[cbred!80, thick, shorten >= 4pt, ->] (y6.center) to ($(0,0)!\epsshift!(y6)$.center);
            \draw[draw=none,fill=cbred!80, thick, ->] ($(0,0)!\epsshift!(y6)$) circle[radius=4pt];

            \draw[fill=cbyellow!90] (y6) circle[radius=5pt];
            \node[black,inner sep=0, below right = 0 pt and -15 pt of y6] () {\scriptsize $f(x^*)$};

            \node[anchor=south west,color=cbdarkblue!]() at (5.5,4.5) {$\apprii{1}$};

        \end{tikzpicture}
        \subcaption*{A halfspace $H$ is chosen and $\hpO(H_{1+\varepsilon})$ identifies a solution $x^*$ with $(1+\varepsilon) f(x^*)\notin H$}
    \end{minipage}

    \vspace{1cm} 

    \begin{minipage}[t]{0.31\textwidth}
        \centering
        \begin{tikzpicture}[scale=0.53]

            \def\xbound{7.3}
            \def\ybound{6.9}
            \def\epsshift{1.3}

            \draw[->, thick] (0,0)--(\xbound,0) node[pos=0.5, below]{$f_1$};
            \draw[->, thick] (0,0)--(0,\ybound) node[pos=0.5, left]{$f_2$};

            \yset

            \node (yupper) at (2.1,\ybound) {};
            \node (yright) at (\xbound,3.5) {};
            \node (ynadir) at (\xbound,\ybound) {};

            \draw[draw=none, fill=cbskyblue!20,] (yupper.center) to (y6.center) to (y3.center)to (yright.center) to (ynadir.center) to cycle;

            \draw[draw=cbskyblue!60, thick,arrows = {Stealth[length=10pt, inset=2pt]-Stealth[length=10pt, inset=2pt]}] (yupper.center) to (y6.center) to (y3.center) to (yright.center);

            \foreach \i in {1,...,9} {
                \draw[fill=black!20] (y\i) circle[radius=5pt];
            }

            \draw[fill=cbskyblue!100] (y3) circle[radius=5pt];
            \draw[fill=cbskyblue!100] (y6) circle[radius=5pt];

            \node[anchor=south west,color=cbdarkblue!]() at (5.5,4.5) {$\apprii{2}$};

        \end{tikzpicture}
        \subcaption*{$x^*$ is added to $R$ and $\apprii{1}$ is updated with $f(x^*)$, resulting in $\apprii{2}$}
    \end{minipage}\hfill
    \begin{minipage}[t]{0.31\textwidth}
        \centering
        \begin{tikzpicture}[scale=0.53]

            \def\xbound{7.3}
            \def\ybound{6.9}
            \def\epsshift{1.3}

            \draw[->, thick] (0,0)--(\xbound,0) node[pos=0.5, below]{$f_1$};
            \draw[->, thick] (0,0)--(0,\ybound) node[pos=0.5, left]{$f_2$};

            \yset

            \foreach \i in {1,...,9} {
                \node (yeps\i) at ($(0,0)!\epsshift!(y\i)$){};
            }

            \node (yupper) at (2.1,\ybound) {};
            \node (yright) at (\xbound,3.5) {};
            \node (ynadir) at (\xbound,\ybound) {};

            \draw[draw=none, fill=cbskyblue!20,] (yupper.center) to (y6.center) to (y3.center)to (yright.center) to (ynadir.center) to cycle;

            \draw[draw=cbskyblue!60, thick,arrows = {Stealth[length=10pt, inset=2pt]-Stealth[length=10pt, inset=2pt]}] (yupper.center) to (y6.center) to (y3.center) to (yright.center);

            \foreach \i in {1,...,9} {
                \draw[cbred!80, thick, shorten >= 4pt, ->] (y\i.center) to ($(0,0)!\epsshift!(y\i)$.center);
                \draw[draw=none,fill=cbred!80, thick, ->] ($(0,0)!\epsshift!(y\i)$) circle[radius=4pt];

                \draw[fill=black!20] (y\i) circle[radius=5pt];
            }

            \draw[draw=cbgreen!100, dashed, ultra thick] ($(yupper)!0.12!(y6)$) to ($(yupper)!2.2!(y6)$);

            \node (mid) at ($(yupper)!1.5!(y6)$) {};

            \draw[draw=cbgreen!100, thick, arrows=-{Latex[width=6pt,length=8pt]}] (mid.center) to ($(mid)!0.75! -90:(y6)$);

            \draw[fill=cbskyblue!100] (y3) circle[radius=5pt];
            \draw[fill=cbskyblue!100] (y6) circle[radius=5pt];

            \node[anchor=south west,color=cbdarkblue!]() at (5.5,4.5) {$\apprii{2}$};

            \node[cbgreen!100,inner sep=0, below right = 0 pt and -4 pt of mid] () {$H$};

        \end{tikzpicture}
        \subcaption*{The next halfspace $H$ is chosen, this time $\hpO(H_{1+\varepsilon})$ returns that all images are approximated by $H$}
    \end{minipage}\hfill
    \begin{minipage}[t]{0.31\textwidth}
        \centering
        \begin{tikzpicture}[scale=0.53]

            \def\xbound{7.3}
            \def\ybound{6.9}
            \def\epsshift{1.3}

            \draw[->, thick] (0,0)--(\xbound,0) node[pos=0.5, below]{$f_1$};
            \draw[->, thick] (0,0)--(0,\ybound) node[pos=0.5, left]{$f_2$};

            \yset

            \node (yupper) at (2.1,\ybound) {};
            \node (yright) at (\xbound,3.5) {};
            \node (ynadir) at (\xbound,\ybound) {};

            \draw[draw=none, fill=cbskyblue!20,] (yupper.center) to (y6.center) to (y3.center)to (yright.center) to (ynadir.center) to cycle;

            \draw[draw=cbskyblue!60, thick,arrows = {Stealth[length=10pt, inset=2pt]-Stealth[length=10pt, inset=2pt]}] (yupper.center) to (y6.center) to (y3.center) to (yright.center);

            \foreach \i in {1,...,9} {
                \draw[fill=black!20] (y\i) circle[radius=5pt];
            }

            \draw[draw=cbgreen!100, dashed, ultra thick] ($(y6)!-1!(y3)$) to ($(y6)!2.9!(y3)$);

            \node (mid) at ($(y6)!1.7!(y3)$) {};

            \draw[draw=cbgreen!100, thick, arrows=-{Latex[width=6pt,length=8pt]}] (mid.center) to ($(mid)!0.9! -90:(y3)$);

            \draw[cbred!80, thick, shorten >= 4pt, ->] (y7.center) to ($(0,0)!\epsshift!(y7)$.center);
            \draw[draw=none,fill=cbred!80, thick, ->] ($(0,0)!\epsshift!(y7)$) circle[radius=4pt];

            \draw[fill=cbskyblue!100] (y3) circle[radius=5pt];
            \draw[fill=cbskyblue!100] (y6) circle[radius=5pt];
            \draw[fill=cbyellow!90] (y7) circle[radius=5pt];

            \node[cbgreen!100,inner sep=0, above right = -3 pt and -1 pt of mid] () {$H$};
            \node[black,inner sep=0, below right = 0 pt and -15 pt of y7] () {\scriptsize $f(x^*)$};

            \node[anchor=south west,color=cbdarkblue!]() at (5.5,4.5) {$\apprii{2}$};

        \end{tikzpicture}
        \subcaption*{The next halfspace $H$ is enumerated and solution $x^*$ returned by $\hpO(H_{1+\varepsilon})$}
    \end{minipage}

    \vspace{1cm} 

    \begin{minipage}[t]{0.31\textwidth}
        \centering
        \begin{tikzpicture}[scale=0.53]

            \def\xbound{7.3}
            \def\ybound{6.9}
            \def\epsshift{1.3}

            \draw[->, thick] (0,0)--(\xbound,0) node[pos=0.5, below]{$f_1$};
            \draw[->, thick] (0,0)--(0,\ybound) node[pos=0.5, left]{$f_2$};

            \yset

            \node (yupper) at (2.1,\ybound) {};
            \node (yright) at (\xbound,1.6) {};
            \node (ynadir) at (\xbound,\ybound) {};

            \draw[draw=none, fill=cbskyblue!20,] (yupper.center) to (y6.center) to (y7.center)to (yright.center) to (ynadir.center) to cycle;

            \draw[draw=cbskyblue!60, thick,arrows = {Stealth[length=10pt, inset=2pt]-Stealth[length=10pt, inset=2pt]}] (yupper.center) to (y6.center) to (y7.center) to (yright.center);

            \foreach \i in {1,...,9} {
                \draw[fill=black!20] (y\i) circle[radius=5pt];
            }

            \draw[fill=cbskyblue!100] (y3) circle[radius=5pt];
            \draw[fill=cbskyblue!100] (y6) circle[radius=5pt];
            \draw[fill=cbskyblue!100] (y7) circle[radius=5pt];

            \node[anchor=south west,color=cbdarkblue!]() at (5.5,4.5) {$\apprii{3}$};

        \end{tikzpicture}
        \subcaption*{$R$ and $\apprii{2}$ are updated again}
    \end{minipage}\hfill
    \begin{minipage}[t]{0.31\textwidth}
        \centering
        \begin{tikzpicture}[scale=0.53]

            \def\xbound{7.3}
            \def\ybound{6.9}
            \def\epsshift{1.3}

            \draw[->, thick] (0,0)--(\xbound,0) node[pos=0.5, below]{$f_1$};
            \draw[->, thick] (0,0)--(0,\ybound) node[pos=0.5, left]{$f_2$};

            \yset

            \foreach \i in {1,...,9} {
                \node (yeps\i) at ($(0,0)!\epsshift!(y\i)$){};
            }

            \node (yupper) at (2.1,\ybound) {};
            \node (yright) at (\xbound,1.6) {};
            \node (ynadir) at (\xbound,\ybound) {};

            \draw[draw=none, fill=cbskyblue!20,] (yupper.center) to (y6.center) to (y7.center)to (yright.center) to (ynadir.center) to cycle;

            \draw[draw=cbskyblue!60, thick,arrows = {Stealth[length=10pt, inset=2pt]-Stealth[length=10pt, inset=2pt]}] (yupper.center) to (y6.center) to (y7.center) to (yright.center);

            \draw[draw=cbgreen!100, dashed, ultra thick] ($(y6)!-1.3!(y7)$) to ($(y6)!1.65!(y7)$);

            \node (mid1) at ($(y6)!0.7!(y7)$) {};

            \draw[draw=cbgreen!100, thick, arrows=-{Latex[width=6pt,length=8pt]}] (mid1.center) to ($(mid1)!1.6! 90:(y7)$);

            \draw[draw=cbgreen!100, dashed, ultra thick] ($(y7)!-0.6!(yright)$) to ($(y7)!0.87!(yright)$);

            \node (mid2) at ($(y7)!0.45!(yright)$) {};

            \draw[draw=cbgreen!100, thick, arrows=-{Latex[width=6pt,length=8pt]}] (mid2.center) to ($(mid2)!0.5! -90:(y7)$);

            \foreach \i in {1,...,9} {
                \draw[cbred!80, thick, shorten >= 4pt, ->] (y\i.center) to ($(0,0)!\epsshift!(y\i)$.center);
                \draw[draw=none,fill=cbred!80, thick, ->] ($(0,0)!\epsshift!(y\i)$) circle[radius=4pt];

                \draw[fill=black!20] (y\i) circle[radius=5pt];
            }

            \draw[fill=cbskyblue!100] (y3) circle[radius=5pt];
            \draw[fill=cbskyblue!100] (y6) circle[radius=5pt];
            \draw[fill=cbskyblue!100] (y7) circle[radius=5pt];

            \node[anchor=south west,color=cbdarkblue!]() at (5.5,4.5) {$\apprii{3}$};

        \end{tikzpicture}
        \subcaption*{$\hpO$ is called for the remaining halfspaces in the next two iterations, both approximate all images}
    \end{minipage}\hfill
    \begin{minipage}[t]{0.31\textwidth}
        \centering
        \begin{tikzpicture}[scale=0.53]

            \def\xbound{7.3}
            \def\ybound{6.9}
            \def\epsshift{1.3}

            \draw[->, thick] (0,0)--(\xbound,0) node[pos=0.5, below]{$f_1$};
            \draw[->, thick] (0,0)--(0,\ybound) node[pos=0.5, left]{$f_2$};

            \yset

            \node (yupper) at (2.1,\ybound) {};
            \node (yright) at (\xbound,1.6) {};
            \node (ynadir) at (\xbound,\ybound) {};

            \draw[draw=none, fill=cbskyblue!20,] (yupper.center) to (y6.center) to (y7.center)to (yright.center) to (ynadir.center) to cycle;

            \draw[draw=cbskyblue!60, thick,arrows = {Stealth[length=10pt, inset=2pt]-Stealth[length=10pt, inset=2pt]}] (yupper.center) to (y6.center) to (y7.center) to (yright.center);


            \draw[fill=cbskyblue!100] (y3) circle[radius=5pt];
            \draw[fill=cbskyblue!100] (y6) circle[radius=5pt];
            \draw[fill=cbskyblue!100] (y7) circle[radius=5pt];

            \node[anchor=south west,color=cbdarkblue!]() at (5.5,4.5) {$\appri$};

        \end{tikzpicture}
        \subcaption*{No unchecked halfspace (and, thus, facet of $\appri$) remains, $R$ and $\appri$ are final}
    \end{minipage}
\end{figure}

%% file: figures/hpoWS.tex
\begin{algorithm}
    \caption{A plane separating oracle for MOMILP using the weighted sum scalarization.
    }\label{algo::innerapx::wsexact}
    \begin{algorithmic}[1]
        \Function{$\hpOWS$}{halfspace $H=\left\{z\in\mathbb{R}^\dobs : w^\T z\geq c\right\}$}
        \State{} $x^*\gets \WS(w)$\label{algo::innerapx::wsexact::wscall}
        \If{$w^\T f(x^*)\geq c$}
            \State{} \Return $\left(\textit{inside}, \emptyset\right)$
        \Else{}
            \State{} \Return $\left(\textit{not inside}, x^*\right)$
        \EndIf{}
        \EndFunction{}
    \end{algorithmic}
\end{algorithm}

%% file: figures/table_representations.tex
\begin{sidewaystable}
\centering
    \caption{Representation quality metrics for \iaa{} and \oaa{}.\label{table::study::metrics}}
{
    \rowcolors{2}{lightgray}{white}
    \begin{tabular}{
        l
        @{\hskip 1em}
        r
        >{\hspace{2.3em}}c
        *{3}{r}
        >{\hspace{2.3em}}c
        *{3}{r}
        >{\hspace{2.3em}}c
        *{3}{r}
        >{\hspace{2.3em}}c
        *{3}{r}
    }
\toprule
\multirow{2}{*}{} & 
\multirow{2}{*}{$n$} & &
\multicolumn{3}{c}{CE} & &
\multicolumn{3}{c}{ME} & &
\multicolumn{3}{c}{HVR} & &
\multicolumn{3}{c}{RR} \\
\cmidrule(lr){4-6} \cmidrule(lr){8-10} \cmidrule(lr){12-14} \cmidrule(lr){16-18}
 & & & \iaa.01 & \iaa.1 & \oaa.1 & & \iaa.01 & \iaa.1 & \oaa.1 & & \iaa.01 & \iaa.1 & \oaa.1 & & \iaa.01 & \iaa.1 & \oaa.1 \\
\midrule
	\cellcolor{white} & 10 & & $.259$ & $.396$ & $.242$ & & $.087$ & $.173$ & $.071$ & & $.944$ & $.775$ & $.967$ & & $.954$ & $.901$ & $.947$ \\
	\cellcolor{white} & 20 & & $.166$ & $.303$ & $.136$ & & $.057$ & $.137$ & $.040$ & & $.944$ & $.776$ & $.971$ & & $.947$ & $.889$ & $.951$ \\
	\cellcolor{white} & 30 & & $.161$ & $.304$ & $.140$ & & $.049$ & $.128$ & $.028$ & & $.946$ & $.792$ & $.977$ & & $.960$ & $.927$ & $.962$ \\
	\cellcolor{white} & 40 & & $.135$ & $.290$ & $.107$ & & $.044$ & $.121$ & $.022$ & & $.949$ & $.803$ & $.981$ & & $.961$ & $.927$ & $.964$ \\
	\multirow{-5}{*}{\cellcolor{white}AP} & 50 & & $.136$ & $.285$ & $.100$ & & $.041$ & $.113$ & $.019$ & & $.953$ & $.822$ & $.984$ & & $.978$ & $.949$ & $.970$ \\
\midrule
	\cellcolor{white} & 20 & & $.383$ & $.625$ & $.308$ & & $.178$ & $.314$ & $.149$ & & $.750$ & $.450$ & $.797$ & & $.998$ & $.736$ & $.030$ \\
	\cellcolor{white} & 40 & & $.377$ & $.647$ & $.249$ & & $.139$ & $.268$ & $.100$ & & $.755$ & $.410$ & $.807$ & & $.976$ & $.584$ & $.988$ \\
	\cellcolor{white} & 60 & & $.307$ & $.658$ & $.206$ & & $.124$ & $.252$ & $.068$ & & $.749$ & $.432$ & $.860$ & & $.937$ & $.640$ & $.956$ \\
	\cellcolor{white} & 80 & & $.309$ & $.629$ & $.166$ & & $.122$ & $.243$ & $.057$ & & $.752$ & $.446$ & $.886$ & & $.962$ & $.549$ & $.975$ \\
	\multirow{-5}{*}{\cellcolor{white}KP} & 100 & & $.279$ & $.630$ & $.157$ & & $.107$ & $.238$ & $.051$ & & $.763$ & $.412$ & $.892$ & & $.945$ & $.590$ & $.981$ \\
\bottomrule
\end{tabular}}

\smallskip

\begin{minipage}{0.84\textwidth}
    \footnotesize
    Selected results for the representation quality of \iaa{} with $\varepsilon=0.01$ (\iaa{}.01), and \iaa{} and \oaa{} with $\varepsilon=0.1$ (\iaa{}.1 and \oaa{}.1, respectively).
For CE and ME, values close to zero indicate a good representation, while for HVR and RR, values close to one indicate a good representation.
\end{minipage}

\end{sidewaystable}

%% file: appendix.tex
\appendix

\section{Generalized Setting}\label{sec::innerapx::generalization}

We outline shortly how to transfer Algorithm~\ref{algo::innerapx} into a more general setting for MOMILP approximation.
This includes, among others, approximating maximization problems.
The new setting consists of two generalizations:

First, instead of all objectives being minimization objectives, each objective can now be either a minimization or a maximization objective.
More formally, let $\opt_{\min},\opt_{\max}\subseteq\mathbb{N}$ form a partition of the index set~$[d]$.
For each~$i\in[d]$, the $i$th objective is then $\min_{x\in X} f_i(x)$ if $i\in\opt_{\min}$, and $\max_{x\in X} f_i(x)$ if $i\in\opt_{\max}$.
The notion of domination also changes accordingly.
Instead of the domination cone being~$\minCone$, the domination cone is now defined as the cone $\genCone \coloneq \cone \left(\left\{e_i : i\in\opt_{\min}\right\}\cup\left\{-e_i : i\in\opt_{\max}\right\}\right)$,
where~$e_i$ denotes the $i$th unit vector of the $d$-dimensional Euclidean space.
For two points $u,v\in \mathbb{R}^{\dobs}$, $u$ dominates~$v$ if $v\in \{u\}+\genCone$ and $u\neq v$.
The assumptions that $X\neq\emptyset$ and $Y\subseteq\mathbb{R}_{\geq0}$ still apply.
Furthermore, for each $i\in\opt_{\max}$, the problem $\max_{x\in X} f_i(x)$ is assumed to be bounded.

\smallskip

Second, instead of $1+\varepsilon$ being a single scalar approximation factor, each objective gets assigned its separate approximation factor.
This allows to vary the precision that is required in each objective.
Let $\varepsilon\in\mathbb{R}_{\geq0}^\dobs$ with $0<\varepsilon_i$ for $i\in\opt_{\min}$ and $0<\varepsilon_i<1$ for $i\in\opt_{\max}$.
Then, let $\Eapx\in\mathbb{R}^{\dobs\times\dobs}$ be a diagonal matrix where, for~$i\in[\dobs]$, $\Eapx_{ii}=1+\varepsilon_i$ if $i\in\opt_{\min}$ and $\Eapx_{ii}=1-\varepsilon_i$ if $i\in\opt_{\max}$.
For two points $u,v\in\mathbb{R}^\dobs$, $u$ \emph{$\Eapx$-approximates}~$v$ if $\Eapx v \in \{u\} + C$.
\emph{$\Eapx$-convex approximation sets} can then be defined analogously to Definition~\ref{def::prelims::convexApx}.

Approximation factors that differ in each objective have a significant impact on the construction of~$\subdi$.
For every~$i\in[\dobs]$, the side-lengths of the hyperrectangles in the $i$th dimension are determined by the exponential function $(1+\varepsilon_i)^k$ if $i\in\opt_{\min}$ and $(1-\varepsilon_i)^{-k}$  if $i\in\opt_{\max}$.
The further $\varepsilon_i$ is away from~$0$, the larger the side-lengths become.
This, in turn, means that fewer hyperrectangles are needed until~$\UB$ (from Definition~\ref{def::prelims::grid}) is covered in the $i$th dimension.

\smallskip

Transferring Algorithm~\ref{algo::innerapx} to the generalized setting requires only minor changes.
In Step~2 and Step~9 of Algorithm~\ref{algo::innerapx}, $\genCone$ is used instead of~$\minCone$ as the recession cone of $\apprii{k}$.
In Step~5, instead of the halfspace~$H_{1+\varepsilon}$, the halfspace $H_\Eapx=\{z\in\mathbb{R}^\dobs : \Eapx w^\T z \geq c\}$ is constructed and used in Step~6.
All our results on the running time and approximation guarantee of Algorithm~\ref{algo::innerapx} then still hold, in particular, Theorems~\ref{theorem::innerapx::Risconvexapx} and~\ref{theorem::innerapx::runningTime}.
The proofs all work analogously to the proofs provided in the previous sections.

\section{Formal Definitions of CE, ME, HVR, and RR.}\label{appendix::sayinmetrics}

All definitions are from~\cite{sayin2024SupportedNondominatedPoints}.
Let $R\subseteq \solfeas$ be a representation for a given multi-objective optimization problem with non-dominated set~$\solnd$.
For $i\in[\dobs]$, we define 
\[\omega_i\coloneqq\frac{1}{\max_{y\in\solnd} y_i - \min_{y\in\solnd} y_i}
\]
and $\dist(y^1,y^2)\coloneqq\max_{i\in[\dobs]} \omega_i\left|y_i^1-y_i^2\right|$.
Then, the \emph{coverage error}~CE is defined as
\(
\text{CE} \coloneqq \max_{y\in\solnd} \min_{r\in R} \dist(y,f(r))
\)
and the \emph{median error}~ME is defined as
\(
\text{ME} \coloneqq \operatorname{median}_{y\in\solnd} \min_{r\in R} \dist(y,r)
\).%

For a set $Y'\subseteq\solnd$, the \emph{hypervolume indicator} $\text{HV}(Y')$ is the volume of the area that is dominated by the points in $Y'$ and bounded by the reference point $p\in\mathbb{R}^d$ with $p_i=\max_{y\in\solnd}y_i + 1$ for every $i\in[\dobs]$.
The \emph{hypervolume ratio}~HVR is then defined as $\text{HVR}\coloneqq\nicefrac{\text{HV}(R)}{\text{HV}(\solnd)}$.%

The \emph{range ratio}~RR is defined as
\[
\text{RR} \coloneqq \operatorname{mean}_{i\in[\dobs]}\frac{\max_{y\in R} y_i - \min_{y\in R} y_i}{\max_{y\in \solnd} y_i - \min_{y\in \solnd} y_i}.
\]